\documentclass{birkjour_t2} 
\usepackage[usenames]{color}
\usepackage{amsmath}
\usepackage{amsthm}
\usepackage{amssymb}
\usepackage{mathrsfs}
\numberwithin{equation}{section}
\usepackage{latexsym}

\usepackage{relsize}

\newtheorem{theorem}{Theorem}[section]

\newtheorem{definition}[theorem]{Definition}

\newtheorem{proposition}[theorem]{Proposition}

\newtheorem{remark}[theorem]{Remark}

\newcommand{\dd}{{\rm d}}
\newcommand{\diver}{{\rm div}}
\newcommand\DT[1]{\mathchoice
{{\buildrel{\hspace*{.1em}\text{\LARGE.}}\over{#1}}}
{{\buildrel{\hspace*{.1em}\text{\Large.}}\over{#1}}}
{{\buildrel{\hspace*{.1em}\text{\large.}}\over{#1}}}
{{\buildrel{\hspace*{.1em}\text{\large.}}\over{#1}}}}

\newcommand{\R}{\mathbb{R}}
\newcommand{\N}{\mathbb{N}}
\renewcommand{\O}{\Omega}
\newcommand{\e}{\varepsilon}

\renewcommand{\a}{\alpha}

\newcommand{\f}{\psi}
\renewcommand{\d}{{\rm d}}

\newcommand{\la}{\left\langle}
\newcommand{\ra}{\right\rangle}
\newcommand{\id}{\mbox{{\rm id}}}
\newcommand{\bulet}{\mbox{\thinspace$_{^{\mbox{$^\bullet$}}}$}}

\newcommand\llangle{\mathchoice{\big\langle\hspace{-.3em}\big\langle}
{\langle\hspace{-.2em}\langle}{\langle\!\langle}{\langle\!\langle}}
\newcommand\rrangle{\mathchoice{\big\rangle\hspace{-.3em}\big\rangle}
{\rangle\hspace{-.2em}\rangle}{\rangle\!\rangle}{\rangle\!\rangle}}

\newcommand{\Gstatic}{G}

\newcommand{\Gm}{{\mathscr G}}
\newcommand{\ccoupl}{\vec{a}}
\newcommand{\mathcalI}{\Theta}

\begin{document}

\begin{sloppypar}

\noindent{\LARGE\bf Thermodynamically-consistent mesoscopic model of\\
the ferro/paramagnetic transition}

\bigskip
\bigskip

\noindent
{\sc B.~Bene\v sov\'a$^{1,4}$, M.~Kru\v z\'\i k$^{2,3}$, T.~Roub\'\i\v cek$^{1,2,4}$}
\bigskip
\bigskip

\noindent
\noindent{\footnotesize
$^1$ Institute of Thermomechanics of the ASCR,
Dolej\v skova 5, CZ-182 00 Praha 8, Czech Republic.\\[-.4em]
$^2$ Institute of Information Theory and Automation of the ASCR,\\[-.4em]
\hspace*{.8em}Pod vod\'arenskou v\v e\v z\'\i\ 4, CZ-182 08 Praha 8,
Czech Republic. \\[-.4em]
$^3$ Faculty of Civil Engineering, Czech Technical University,
Th\'{a}kurova 7, 166 29 Praha 6, Czech Republic.\\[-.4em]
$^4$ Mathematical Institute, Charles University,
Sokolovsk\'a 83, CZ-186~75~Praha~8, Czech Republic.
}

\bigskip\begin{center}
\begin{minipage}[t]{.85\textwidth}
{\small
\baselineskip=9pt
\noindent{\it Abstract}:
A continuum evolutionary model for micromagnetics is presented that, beside the
standard magnetic balance laws, includes thermo-magnetic coupling.
To allow conceptually efficient computer implementation, inspired by
relaxation method of static minimization problems, our model is
mesoscopic in the sense that possible fine spatial oscillations of the
magnetization are modeled by means of Young measures. Existence of weak
solutions is proved by backward Euler time discretization.
\medskip

\noindent{\it AMS Subj. Classification}:
35K85, 
35Q60, 
49S05, 
78A30, 
78M30, 
80A17. 
\medskip

}
\end{minipage}
\end{center}

\bigskip\bigskip

\section{Introduction, static problem and its relaxation}
Micromagnetics is a continuum theory introduced by Brown \cite{brown3}
to describe the equilibrium states of saturated ferromagnets. The equilibria are determined as minimizers of a functional with exchange, anisotropy, Zeeman (interaction), and magnetostatic energy contributions. This theory also predicts the formation of domain structures. The reader is referred to \cite{kruzik-prohl} for a recent survey on the topic. In the isothermal situation, the configuration of a rigid ferromagnetic body occupying a bounded domain $\O\subset\R^d$ is typically described by a magnetization $m:\O\to\R^d$ which vanishes if the temperature $\theta$ is above the so-called {\it Curie temperature} $\theta_{\rm c}$ and no external magnetic field is applied.

On {\it microscopic level}, the magnetic {\it Gibbs energy} consists of
five parts, namely an {\it anisotropy energy} $\int_\O\psi(m,\theta)\,\d x$, an {\it exchange energy} $\frac12\int_\O\e|\nabla m(x)|^2\d x$ having a quantum-theoretical origin, the non-local {\it magnetostatic energy} $\frac{1}{2}\int_{\R^d}\!\mu_0|\nabla u_m(x)|^2\d x =\frac12\int_\O m{\cdot}\nabla u_m\,\d x$, an {\it interaction energy} $-\int_\O h(x)\!\cdot\!m(x)\,\d x$ involving the outer magnetic field $h$ and, finally, a \emph{calorimetric term} $\int_\O \psi_0 \, \d x$. In the anisotropic energy, we denoted $\psi$ its density that will depend on the material properties and should exhibit crystallographic symmetry. Furthermore, $\psi$ is supposed to be a nonnegative function, even in its first variable. In the magnetostatic energy, $u_m$ is the magnetostatic potential related to $m$ by $\mbox{div}(\mu_0\nabla u_m{-}\chi_\O m)=0$ arising from simplified
magnetostatic Maxwell equations. Here $\chi_\O:\R^d\to\{0,1\}$ denotes
the characteristic function of $\O$ and $\mu_0$ is the permeability of vacuum. Moreover, it shall be noted that, due to the quantum mechanical origin, $\e$ in the exchange energy is very small from the macroscopic point of view. 

A widely accepted model for steady-state isothermal configurations is due to Landau and Lifshitz \cite{landau-lifshitz,landau-lifshitz-1} (see e.g. Brown \cite{brown3} or Hubert and Sch\"afer \cite{hubert-schafer}), relying on a minimum-of-Gibbs'-energy principle with $\theta$ as a parameter, i.e.
\begin{align}\label{micromagnetics}
\left.\begin{array}{ll}
\mbox{minimize } & \Gstatic_\e(m)
:=\displaystyle{\int_\O \Big(\psi(m,\theta)+\frac12 m{\cdot}\nabla u_m
+\frac\e2 |\nabla m|^2 - h{\cdot}m \,\d x\Big)\,\d x}\\[3mm]
\mbox{subject to } & \mbox{div}(\mu_0\nabla u_m-\chi_\O m)=0 \ \mbox{ in } \R^d \ ,\\[2mm] & m\in H^1(\O;\R^d),\ \ u_m\in H^1(\R^d),
\end{array}\right\}
\end{align}
where the anisotropy energy $\psi$ is considered in the form
\begin{equation}
\psi(m,\theta):=\phi(m)+a_0(\theta-\theta_{\rm c})|m|^2-\psi_0(\theta),
\label{def-of-psi-I}
\end{equation}
where $a_0$ determines the intensity of the thermo-magnetic-coupling. To see a
paramagnetic state above Curie temperature $\theta_{\rm c}$, one should
consider $a_0>0$. The \emph{isothermal part of the anisotropy energy density} $\phi:\R^d\to [0,\infty)$ typically consists of two components $\phi(m) = \phi_\mathrm{poles}(m) + b_0|m|^4$, where $\phi_\mathrm{poles}(m)$ is chosen in such a way to attain its minimum value (typically zero) precisely on lines $\{ ts_\a;\ t\in\R\}$, where each $s_\a\in\R^d$, $|s_\a|=1$ determines an {\it axis of easy magnetization}. Typical examples are $\a=1$ for uni-axial magnets and $1\le\a\le 3$ or $1\le\a\le 4$ for cubic magnets. On the other hand, $b_0|m|^4$ is used to assure that, for $\theta<\theta_{\rm c}$, $\f(\cdot,\theta)$ is minimized at $ts_\a$ for $|t|^2=(\theta_{\rm c}-\theta)a_0/(2b_0)$. Such energy has already been used in \cite{podio-roubicek-tomassetti}. For $\e>0$, the exchange energy $\e|\nabla u|^2$ guarantees that the problem \eqref{micromagnetics} has a solution $m_\e$. Zero-temperature limits of this model consider, in addition, that the minimizers to \eqref{micromagnetics} are constrained to be valued on the sphere with the radius $\sqrt{a_0\theta_{\rm c}/(2b_0)}$ and were investigated, e.g., by Choksi and Kohn \cite{choksi-kohn}, DeSimone \cite{desim}, James and Kinderlehrer \cite{jam-kin}, James and M\"uller \cite{jam-mu}, Pedregal \cite{pedregal0,pedregal}, Pedregal and Yan \cite{PedregalYan} and many others.

For $\e$ small, minimizers $m_\e$ of \eqref {micromagnetics} typically exhibit fast spatial oscillations, a so-called {\it fine structure}. Indeed, the anisotropy energy, which forces magnetization vectors to be aligned with the easy axis (axes), competes with the magnetostatic energy preferring divergence-free magnetization fields. If the exchange energy term is neglected, and this is a justified simplification of the functional for large ferromagnets \cite{desim}, nonexistence of a minimum for uniaxial ferromagnets can be expected and was shown in \cite{jam-kin} for the zero external field $h$ and zero temperature. Hence, various concepts of relaxation (in the sense of variational calculus) were introduced in order to cope with this phenomenon. The idea is to capture the limiting behavior of minimizing sequences of $\Gstatic_\e(m)$ as $\e \to 0$. This leads to a relaxed problem \eqref{relaxedmicromagnetics} involving so-called Young measures $\nu$'s which describe the relevant ``mesoscopical'' character of the fine structure of $m$. We call this ``limit'' a {\it microstructure}.

It can be proved \cite{desim,pedregal0} that this limit configuration $(\nu,u_m)$ solves the following minimization problem involving ``mesoscopical'' Gibbs' energy $\Gstatic$:
\begin{align}\label{relaxedmicromagnetics}
\left.\begin{array}{llr}
\mbox{minimize} &
\displaystyle{\Gstatic(\nu,m) :=\int_\O \big(\f\bulet\nu+\frac12 m{\cdot}\nabla u_m-h{\cdot}m \big)\,\d x}\ \ \ \\[2mm]
\mbox{subject to} & \displaystyle{\mbox{div} \big(\mu_0\nabla u_m-\chi_\O m \big)=0} \ \ \ \mbox{ on } \R^d,\\[2mm] & m=\id \bulet\nu\hspace{7.4em}\mbox{ on } \O,\\[2mm]
& \nu\!\in\! {\mathscr Y}^p(\O;\R^d),\ \ \ m\!\in\! L^p(\O;\R^d), \ \ \ u_m\!\in\! H^1(\R^d)\ ,
\end{array}\right\}\hspace*{-1cm}
\hspace*{-.2cm}
\end{align}
where the ``momentum'' operator ``$\,\bulet\,$'' is defined by $[f\bulet\nu](x):=\int_{\R^d}f(s)\nu_x(\d s)$, $\id:\R^d\to\R^d$ denotes the identity and $\nu \in {\mathscr Y}^p(\O;\R^d)$. Here, the set of {\it Young measures} ${\mathscr Y}^p(\O;\R^d)\subset L_{\rm w}^\infty(\O;\mathcal{M}(\R^d))\cong L^1(\O;C_0(\R^d))^*$ is the set of all weakly measurable essentially bounded mappings $x\mapsto\nu_x:\O\to \mathcal{M}(\R^d)\cong C_0(\R^d)^*$ such that $\int_\O\int_{\R^d}|s|^{p}\nu_x({\rm d}s){\rm d}x<+\infty$; here $C_0(\R^d)$ denotes the set of continuous functions with compact support and thus $\mathcal{M}(\R^d)$ is the set of Radon measures on $\R^d$, and the adjective ``weakly measurable'' means that $v\bulet\nu$ is Lebesgue measurable for any $v\in C_0(\R^d)$. A natural embedding $i:L^{p}(\O;\R^d)\to{\mathscr Y}^p(\O;\R^d)$ of a magnetization $m$ is a Young measure $\nu=i(m)$ defined by $\nu_x=\delta_{m(x)}$ with $\delta_s$ denoting the Dirac measure at $s\in \R^d$. We say that a sequence $\{\nu^k\}_{k\in\N}\subset{\mathscr Y}^p(\O;\R^d)$ converges weakly* to $\nu$ if
\begin{align}\label{Ym-def}
\forall f\!\in\!L^1(\O;C_0(\R^d)):\ \ \ \ \lim_{k\to\infty}\int_\O f\bulet\nu^k\d x\ =\ \int_\O f\bulet\nu\,\d x
\end{align}
The set ${\mathscr Y}^{p}(\O;\R^d)$ is convex, metrizable, compact, and contains weak*-densely the set of magnetizations $m\in L^{p}(\O;\R^d)$ if embedded via $i$. One can thus call ${\mathscr Y}^{p}(\O;\R^d)$ a {\it convex compactification} of the set of admissible magnetizations, cf.~also \cite[Chapter~3]{roubicek0}. The so-called {\it relaxed problem} \eqref{relaxedmicromagnetics} can then be understood simply as a continuous extension of the original problem \eqref{micromagnetics} considered for $\e=0$. Let us note that the problem \eqref{relaxedmicromagnetics} has a convex structure. Moreover, it captures the {\it multiscale} character of the problem. The Young measure solving \eqref{relaxedmicromagnetics} encodes limiting oscillating behavior of minimizing sequences of \eqref{micromagnetics} while its first moment, the magnetization, resolves the macroscopic magnetization $m$. Due to these properties we call the model in \eqref{relaxedmicromagnetics} {\it mesoscopic}. An equivalent way how to relax \eqref{micromagnetics} is to replace $\f$ by its convex envelope. The drawback for numerical calculations is that one needs to know the convex envelope explicitly. There were many attempts to design numerical schemes for both \eqref{micromagnetics}, as well as for \eqref{relaxedmicromagnetics} in the zero-temperature situation; cf.~e.g.~\cite{carstensen-prohl,kruzik-prohl-1,luskin-ma}. Departing from \eqref{relaxedmicromagnetics} and following the ideas of \cite{mielke-theil-levitas}, a model of the isothermal rate-independent evolution exhibiting hysteretic response was proposed and analyzed in
\cite{roubicek-kruzik-1,roubicek-kruzik-2}.

Our motivation is to merge concepts of relaxation that can successfully fight with multiscale character of the problem with recent ideas to build thermodynamically consistent mesoscopic models in anisothermal situations. A closely related thermodynamically consistent model on the microscopic level was introduced in \cite{podio-roubicek-tomassetti} to ferro/para magnetic transition. Another related microscopic model with a prescribed temperature field was investigated in \cite{banas-prohl-slodicka}. The goal is to develop a model that would be supported by rigorous analysis and would allow for computationally efficient numerical implementation like in \cite{kruzik-prohl-1,kruzik-roubicek,kruzik-roubicek-2} where such a model was used in the isothermal variant.

This paper is structured as follows: In Section 2 we introduce the concept of general standard materials and handle thermodynamics, in Section 3 we give a weak formulation of the model equations, in Section 4 we give a time-discretization of the model equations and prove a priori estimates, in Section 5 we then prove convergence of the discrete solutions. Finally in Section 6, we present a generalization of the model, where we allow for a weaker dissipation so that at least some internal parameters can be considered evolving
in a purely rate-independent manner.

\section{Evolution problem, dissipation, mesoscopic Gibbs free energy}\label{sect-evol}
If the external magnetic field $h$ varies during a time interval $[0,T]$, $T>0$, the energy of the system as well as the magnetization evolve, too. Change of magnetization may cause energy dissipation. As the magnetization is the first moment of the Young measure we relate the dissipation on the mesoscopic level to temporal changes of some moments of $\nu$ and consider these moments as separate variables. This approach was already used in micromagnetics in \cite{roubicek-kruzik-1,roubicek-kruzik-2} and proved to be useful also in modeling of dissipation in shape memory materials, see e.g.~\cite{mielke-roubicek}. In view of \eqref{def-of-psi-I}, we restrict ourselves to the first two moments defining $\lambda=(\lambda_1,\lambda_2)$ giving rise to the constraint
\begin{equation}
\lambda=L\bulet\nu \ ,\qquad \text{ where }\ L(m):=(m,|m|^2)
\label{moment-constraint}
\end{equation}
and consider the specific dissipation potential
\begin{equation}\label{form-of-zeta}
\zeta(\DT{\lambda}):= \delta_S^*(\DT{\lambda})+\frac\epsilon{q}|\DT{\lambda}|^q,\qquad q\ge2,
\end{equation}
where $\delta_S^*:\R^{d+1}\to\R^+$ is the Legendre-Fenchel conjugate function to the indicator function $\delta_S:\R^{d+1}\to\{0,+\infty\}$ to a convex bounded neighborhood of 0.
The set $S$ determines activation threshold for the evolution of $\lambda$.
The function $\delta_S^*\ge0$ is convex and degree-1 positively homogeneous with $\delta_S^*(0)=0$. In fact, the first term describes purely hysteretic losses, which are rate-independent (the contribution of which we consider dominant) and the second term models rate-dependent dissipation.

In view of \eqref{def-of-psi-I}--\eqref{relaxedmicromagnetics}, the specific mesoscopic Gibbs free energy, expressed in terms of $\nu$, $\lambda$ and $\theta$, reads as
\begin{subequations}\begin{align}
g(t,\nu,\lambda,\theta):=\phi\bulet \nu+ (\theta{-}\theta_{\rm c}) \ccoupl{\cdot}\lambda-\psi_0(\theta)+\frac{1}{2}m{\cdot}h_{\rm dem}- h(t){\cdot}m \label{gibbs}\\
\qquad\text{with }\ m=\id \bulet \nu\quad \text{and}\quad h_{\rm dem}=\nabla u_{m}, \end{align}
\end{subequations}
where we denoted $\ccoupl :=(0,\ldots,0,a_0)$ with $a_0$ from \eqref{def-of-psi-I} and, of course, $u_m$ again from \eqref{micromagnetics},which makes $g$ non-local.

In what follows, we relax the constraint \eqref{moment-constraint} by augmenting the total Gibbs free energy (i.e. $g$ integrated over $\O$) by the term $\frac\varkappa2\|\lambda-L\bulet\nu\|^2_{H^{-1}(\O;\R^{d+1})}$ with (presumably large) $\varkappa\in\R^+$ and with $H^{-1}(\O;\R^{d+1})\cong H^1_0(\O;\R^{d+1})^*$. Thus, $\lambda$'s no longer exactly represent the ``macroscopical'' momenta of the magnetization but rather are in a position of a \emph{phase field}.
Let us choose a specific norm on $H^{-1}(\O;\R^{d+1})$ as $\|f\|_{H^{-1}(\O;\R^{d+1})}:=\|\chi\|_{L^2(\O;\R^{d+1})}$ with $\chi\in L^2_{\rm grad}(\O;\R^{d})$ defined uniquely by ${\rm div}\,\chi=\lambda{-}L\bulet\nu$ where $L^2_{\rm grad}(\O;\R^{d}):=\{\nabla v;\ v\in H^1_0(\O;\R^{d+1})\}$. In other words, $\|f\|_{H^{-1}(\O;\R^{d+1})}=\|\nabla\Delta^{-1}f\|_{L^2(\O;\R^{d+1})}$ with $\Delta^{-1}$ meaning the inverse of the Dirichlet boundary-value problem for the Laplacean $\Delta:H^1_0(\O;\R^{d+1})\to H^{-1}(\O;\R^{d+1})$ .
We define the mesoscopic Gibbs free energy $\Gm$ by
\begin{align}\label{def-of-psi}
\hspace{-3ex}\Gm(t,\nu,\lambda, \theta):=\!\int_\O\! \Big(g(t,\nu,\lambda,\theta)+\frac\varkappa2|\nabla\Delta^{-1} (\lambda - L \bulet \nu)|^2 \Big)\,\d x 
\end{align}

Notice that the $H^{-1}$norm defined above is equivalent to the standard $H^{-1}$-norm defined as $\|f\|_{-1}:=\sup_{v\in W^{1,2}_0(\O;\R^{d+1}),\ v\ne 0}\la f,v\ra/\|v\|_{W^{1,2}_0(\O;\R^{d+1})}$. Indeed, if $h\in W^{1,2}_0(\O;\R^{d+1})$ solves $-\Delta h=f$ and $\|v\|_{W^{1,2}_0(\O;\R^{d+1})}:=\|\nabla v\|_{L^2(\O;\R^{d+1})}$ then
$$
\|f\|_{-1}= \sup_{v\in W^{1,2}_0(\O;\R^{d+1}),\ v\ne 0}\|v\|^{-1}_{W^{1,2}_0(\O;\R^{d+1})}\int_\O\nabla h\cdot\nabla v\,\d x\ge \beta\|\nabla h\|_{L^2(\O;\R^{d+1})}\ ,$$
where $\beta>0$ is the ellipticity constant of $B[h,v]:=\int_\O\nabla h\cdot\nabla v\,\d x$. On the other hand,
$$
\|f\|_{-1}=\sup_{v\in W^{1,2}_0(\O;\R^{d+1}),\ v\ne 0}\|v\|^{-1}_{W^{1,2}_0(\O;\R^{d+1})}\int_\O\nabla h\cdot\nabla v\,\d x\le \|\nabla h\|_{L^2(\O;\R^{d+1})}$$ by the Cauchy-Schwartz inequality.

We shall give a certain justification of the penalized model at the end of the section. Yet, we should
emphasize that we will consider $\varkappa$ fixed thorough this article.

The value of the internal parameter may influence the magnetization of the system (and vice versa) and, on the other hand, dissipated energy may influence the temperature of the system, which, in turn, may affect the internal parameter. In order to capture these effects we employ the concept of general standard materials \cite{halphen-nguyen} known from continuum mechanics and couple our micromagnetic model with the entropy balance with the rate of dissipation on the right-hand side; cf.\ \eqref{EntropyEquation}.
Then $\nu$ is considered evolving purely quasistatically
according to the  minimization principle of the Gibbs energy $\mathcal G(t,\cdot,\lambda,\theta)$ while the ``dissipative'' variable $\lambda$ is considered as governed by the {\it flow rule} in the form:
\begin{align}\label{flow-rule-physically}
\partial\zeta(\DT\lambda)=\partial_\lambda g(t,\nu,\lambda,\theta)
\end{align}
with $\partial \zeta$ denoting the subdifferential of the convex functional $\zeta(\cdot)$ and similarly $\partial_\lambda g$ is the subdifferential of the convex functional $g(t,\nu,\cdot,\theta)$. In our specific choice, \eqref{flow-rule-physically} takes the form
$\partial\delta_S^*(\DT{\lambda})+\epsilon|\DT{\lambda}|^{q-2}\DT{\lambda}+(\theta{-}\theta_{\rm c})\ccoupl\ni \varkappa\Delta^{-1}(\lambda - L \bulet \nu)$.
Furthermore, we define the specific {\it entropy} $s$ by the standard Gibbs relation for entropy, i.e.\ $s=-g'_\theta(t,\nu,\lambda,\theta)$, and write the {\it entropy equation}
\begin{align}\label{EntropyEquation}
\theta\DT s+{\rm div}\,j=\xi(\DT\lambda)=\,\text{ heat production rate},
\end{align}
where $j$ is the heat flux. In view of \eqref{form-of-zeta},
\begin{align}
\xi(\DT\lambda)=\partial\zeta(\DT\lambda){\cdot}\DT\lambda
=\delta_S^*(\DT{\lambda})+\epsilon|\DT{\lambda}|^q.
\end{align}
We assume $j$ is the heat flux assumed governed by the {\it Fourier law}
\begin{align}
j=-\mathbb{K}\nabla\theta
\end{align}
with a {\it heat-conductivity} tensor $\mathbb{K}=\mathbb{K}(\lambda,\theta)$.
Now, since $s=-g'_\theta(t,\nu,\lambda,\theta)=-g'_\theta(\lambda,\theta)$,
it holds $\theta\DT s=-\theta g''_\theta(\lambda, \theta)\DT{\theta}-\theta g''_{\theta\lambda} \DT{\lambda}$. Using also $g''_{\theta\lambda}=\ccoupl$, we may reformulate the entropy equation \eqref{EntropyEquation} as {\it the heat equation}
\begin{equation}
c_\mathrm{v}(\theta) \DT{\theta}-\diver(\mathbb{K}(\lambda,\theta)\nabla \theta) =\delta_S^*(\DT{\lambda}) + \epsilon|\DT{\lambda}|^q+\ccoupl{\cdot}\theta \DT{\lambda}
\quad\text{ with }\
c_\mathrm{v}(\theta)=-\theta g ''_\theta(\theta),
\label{heatEquation0}
\end{equation}
where $c_\mathrm{v}$ is the specific \emph{heat capacity}.

Altogether, we can formulate our problem as
\begin{subequations} \label{ContinuousSystem}
\begin{align}
\label{ContinuousSystem-min}
&\!\!\!\!\!\!
\left.\begin{array}{llr}
\mbox{minimize} &
\displaystyle{
\int_\O\!\!\Big(\phi\bulet \nu {+} (\theta(t)){-}\theta_{\rm c}) \ccoupl{\cdot}\lambda(t){-}\psi_0(\theta(t)) + \frac{1}{2}
m{\cdot}h_{\rm dem}} \\& \qquad - h(t){\cdot}m
+\frac\varkappa2\big|\nabla\Delta^{-1}(\lambda(t){-}L\bulet\nu)\big|^2
\Big)\,\d x\!\!\!\!\!\!
\\[2mm]
\mbox{subject to}\!& m=\id \bulet\nu,\ \ \ \ h_{\rm dem}=\nabla u_{m}
\hspace{2.4em}\mbox{ on } \O,\\[2mm]
&
\displaystyle{\mbox{div} \big(\mu_0\nabla u_m-\chi_\O m \big)=0}
\hspace{3.4em}\mbox{ on } \R^d,\\[2mm] &
\nu\!\in\! {\mathscr Y}^p(\O;\R^d),\ m\!\in\! L^p(\O;\R^d),
\ u_m\!\in\! H^1(\R^d),
\end{array}\hspace*{.5em}\right\}
\text{for }t\!\in\![0,T],
\hspace*{-1em}
\\
&\partial\delta_S^*(\DT{\lambda})
+\epsilon|\DT{\lambda}|^{q-2}\DT{\lambda}+(\theta{-}\theta_{\rm c})\ccoupl\ni
\varkappa\Delta^{-1}({\rm div}\,\lambda-L\bulet \nu)
\hspace{3.2em}\text{in $Q:=[0,T]{\times}\Omega$,} \label{FlowRuleBasic}\\
&c_{\rm v}(\theta)\DT\theta-\diver(\mathbb{K}(\lambda,\theta)\nabla \theta)
= \delta_S^*(\DT{\lambda}) + \epsilon|\DT{\lambda}|^q
+\ccoupl{\cdot} \DT{\lambda} \theta
\qquad\text{in $Q$,} \label{heatEquation}\\
& \big(\mathbb{K}(\lambda,\theta)\nabla\theta\big){\cdot}n + b\theta=
{b}\theta_{\mathrm{ext}}\hspace{10.4em}
\text{ on $\Sigma:= [0,T]{\times}\Gamma$,}
\label{boundaryCondTemp}
\end{align}
\end{subequations}
where we accompanied the heat equation \eqref{heatEquation} by the Robin-type boundary conditions with $n$ denoting the outward unit normal to the boundary $\Gamma$, and with $b\in L^\infty(\Gamma)$ a phenomenological heat-transfer coefficient and $\theta_{\mathrm{ext}}$ an external temperature, both assumed non-negative.

Next we shall transform \eqref{heatEquation} by a so-called \emph{enthalpy transformation}, which simplifies the analysis below. For this, let us introduce a new variable $w$, called \emph{enthalpy}, by
\begin{equation}
w=\widehat c_\mathrm {v}(\theta)=\int_0^\theta c_\mathrm{v}(r) \dd r,
\label{IntroduceEnthalpy}
\end{equation}
It is natural to assume $c_\mathrm{v}$ positive, hence $\widehat c_{\rm v}$ is, for $w \geq 0$ increasing and thus invertible. Therefore, denote
 $$
\mathcalI(w):=\begin{cases}
\widehat c_{\rm v}^{-1}(w) & \text{if $w \geq 0$} \\
0 & \text{if $w < 0$}
\end{cases}
$$ and note that, in the physically relevant case when $\theta \geq 0$, $\theta = \mathcalI(w)$. Thus writing the heat flux in terms of $w$ gives
\begin{align}\label{flux-in-w}
\mathbb{K}(\lambda,\theta)\nabla\theta =\mathbb{K}\big(\lambda,\mathcalI(w)\big)\nabla\mathcalI(w)
=\mathcal{K}(\lambda,w)\nabla w \quad\text{ where }\ \mathcal{K}(\lambda,w):=
\frac{\mathbb{K}(\lambda, \mathcalI(w))}{c_\mathrm{v}(\mathcalI(w))}.
\end{align}
Moreover, the terms $(\mathcalI(w(t)){-}\theta_{\rm c}) \ccoupl{\cdot}\lambda(t)$ and $\psi_0(\theta(t))$ obviously do not play any role in the minimization \eqref{ContinuousSystem-min} and can be omitted. 
Thus we may rewrite \eqref{ContinuousSystem-min} in terms of $w$ as follows:
\begin{subequations} \label{ContinuousSystem1}
\begin{align}\label{ContinuousSystem1+}
&\!\!\!\!\!\!
\left.\begin{array}{llr}
\mbox{minimize} &
\displaystyle{
\int_\O\!\! \Big(\phi\bulet \nu +\frac{1}{2}
m{\cdot}h_{\rm dem}}- h(t){\cdot}m
+\frac\varkappa2\big|\nabla\Delta^{-1}(\lambda(t){-}L\bulet\nu)\big|^2
\Big) \,\d x\!\!\!\!\!\!
\\[2mm]
\mbox{subject to}\!& m=\id \bulet\nu,\ \ \ \ h_{\rm dem}=\nabla u_{m}
\hspace{2.4em}\mbox{ on } \O,\\[2mm]
&
\displaystyle{\mbox{div} \big(\mu_0\nabla u_m-\chi_\O m \big)=0}
\hspace{3.4em}\mbox{ on } \R^d,\\[2mm] &
\nu\!\in\! {\mathscr Y}^p(\O;\R^d),\ m\!\in\! L^p(\O;\R^d),
\ u_m\!\in\! H^1(\R^d),
\end{array}\hspace*{.5em}\right\}
\text{for }t\!\in\![0,T],
\hspace*{-1em}
\\
&\partial\delta_S^*(\DT{\lambda})
+\epsilon|\DT{\lambda}|^{q-2}\DT{\lambda}
+\big(\mathcalI(w){-}\theta_{\rm c}\big)\ccoupl\ni
\varkappa\Delta^{-1}(\lambda{-}L\bulet\nu)
\hspace{8.6em}\text{in $Q$,}
\label{FlowRuleBasic1}
\\
&\DT{w} - \diver(\mathcal{K}(\lambda, w) \nabla w)
=\delta_S^*(\DT{\lambda}) + \epsilon|\DT{\lambda}|^q
+\ccoupl \cdot \mathcalI(w)\DT{\lambda}
\hspace{9.6em}\text{in $Q$,} \label{EnthalpyEquation}
\\
&
\big(\mathcal K(\lambda,w)\nabla w\big){\cdot}n
+b\mathcalI(w)=b\theta_{\mathrm{ext}}
\hspace{16.6em}\text{ on $\Sigma$}.
\label{boundaryCondEnthalp}
\end{align}
\end{subequations}
Eventually, we complete this transformed system by the initial conditions
\begin{align}
\nu(0,\cdot)=\nu_0,\qquad\lambda(0,\cdot)=\lambda_0,\qquad
w(0,\cdot)=w_0:=\widehat c_{\rm v}(\theta_0) \qquad\ \text{ on }\ \Omega,
\label{initCond}
\end{align}
where $(\nu_0,\lambda_0)$ is the initial microstructure assumed to solve \eqref{ContinuousSystem1+}
and the phase field, and $\theta_0$ is the initial temperature. Note also that,
by prescribing $\nu_0$, we also prescribe
the initial magnetization $m_0=\id \bulet\nu_0$ and
magnetic potential $u_{m_0}$.

{
\subsection{Justification of the penalization concept}
Recall that in the model \eqref{ContinuousSystem} we gave up the constraint $\lambda = L \bulet \nu$ and only included a penalization term $\frac \varkappa2 \|\lambda - L \bulet \nu\|_{H^{-1}(\Omega; \R^{d+1})}$ in the Gibbs free energy. To justify this approach, we show that in some particular situations, namely in the static case and also in the iso-thermal rate-independent (with a small modification) case, solutions of the penalized model converge to solutions of the original model that satisfy $\lambda = L \bulet \nu$ as $\varkappa \to \infty$. We shall also give some heuristic ideas, why a similar limit passage should be possible even in \eqref{ContinuousSystem}, however a rigorous proof is beyond the scope of this paper.

\noindent \emph{The static case:}

Let us consider an analogical problem to \eqref{relaxedmicromagnetics} that includes also a penalization term and where we also use the form of the Gibbs free energy as in \eqref{gibbs} with $\theta$ given, i.e.
\begin{align}\label{PenalizedStatic}
\hspace{-1cm}\left.\begin{array}{llr}
\mbox{minimize} &
\displaystyle \int_\O \!\!\Big(\phi\bulet \nu+ (\theta{-}\theta_{\rm c}) \ccoupl{\cdot}\lambda-\psi_0(\theta)+\frac{1}{2}m{\cdot}\nabla u_m{-} h(t){\cdot}m {+}\frac{\varkappa}{2}|\nabla \Delta^{-1}(\lambda-L\bulet \nu)|^2\big)\,\d x\\[2mm]
\mbox{subject to} & \displaystyle{\mbox{div} \big(\mu_0\nabla u_m-\chi_\O m \big)=0} \ \ \ \mbox{ on } \R^d,\\[2mm] & m=\id \bulet\nu\hspace{7.4em}\mbox{ on } \O,\\[2mm]
& \nu\!\in\! {\mathscr Y}^p(\O;\R^d),\ \ \ \lambda \in H^{-1}(\Omega; \R^{d+1}),\ \ \ m\!\in\! L^p(\O;\R^d), \ \ \ u_m\!\in\! H^1(\R^d).
\end{array}\hspace{-0.3cm}\right\}\hspace*{-0.6cm}
\hspace*{-.2cm}
\end{align} 

Let us denote $(\lambda_\varkappa, \nu_\varkappa) \in  H^{-1}(\Omega; \R^{d+1})\times {\mathscr Y}^p(\O;\R^d)$ the solutions to \eqref{PenalizedStatic}. Let us then show that they (in terms of a subsequence) converge weakly* in $H^{-1}(\Omega; \R^{d+1})\times{\mathscr Y}^p(\O;\R^d)$ to the solutions of 
\begin{align}\label{PenalizedStaticConverges}
\hspace{-1.2cm}\left.\begin{array}{llr}
\mbox{minimize} &
\displaystyle \int_\O \Big(\phi\bulet \nu+ (\theta{-}\theta_{\rm c}) \ccoupl{\cdot}\lambda-\psi_0(\theta)+\frac{1}{2}m{\cdot}\nabla u_m- h(t){\cdot}m\Big)\,\d x\ \ \ \\[2mm]
\mbox{subject to} & \displaystyle{\mbox{div} \big(\mu_0\nabla u_m-\chi_\O m \big)=0} \ \ \ \mbox{ on } \R^d,\\[2mm] & m=\id \bulet\nu, \, \lambda = L \bulet \nu\hspace{2.9em}\mbox{ on } \O,\\[2mm]
& \nu\!\in\! {\mathscr Y}^p(\O;\R^d), \ \ \ \lambda \in H^{-1}(\Omega; \R^{d+1}),\ \ \ m\!\in\! L^p(\O;\R^d), \ \ \ u_m\!\in\! H^1(\R^d).\ \ \ \
\end{array}\hspace{-0.8cm}\right\}\hspace*{-1cm}
\hspace*{-.1cm}
\end{align} 

Namely, it easy to see, from coercivity of $\phi$ and by simply testing \eqref{PenalizedStatic} by any $(\hat{\lambda}, \hat{\nu})$ such that $\hat{\lambda} = L \bulet \hat{\nu}$, that $\int_\Omega |\cdot|^p \bulet \nu_\varkappa \dd x $ is bounded uniformly with respect to $\varkappa$; here $p$ corresponds to the growth of $\phi$, cf. (\ref{ass}a). Hence also $\|L\bulet \nu_\varkappa\|_{L^2(\Omega; \R^{d+1})}$ and in turn also $\|L\bulet \nu_\varkappa\|_{H^{-1}(\Omega; \R^{d+1})}$ are uniformly bounded with respect to $\varkappa$.

Also, by the same test as above, we get that $\varkappa \|\lambda_\varkappa - L\bulet \nu_\varkappa\|_{H^{-1}(\Omega; \R^{d+1})}^2$ and thus also $\|\lambda_\varkappa\|_{H^{-1}(\Omega; \R^{d+1})}^2$ are bounded independently of $\varkappa$.

Hence, exploiting standard selection principles, we find a pair $(\nu, \lambda) \in {\mathscr Y}^p(\O;\R^d) \times H^{-1}(\Omega; \R^{d+1})$ such that (in terms of a not-relabeled subsequence) $\nu_\varkappa \stackrel{*}{\rightharpoonup} \nu$ in ${\mathscr Y}^p(\O;\R^d)$ and $\lambda_\varkappa \rightharpoonup \lambda$ in $H^{-1}(\Omega; \R^{d+1})$. Also, as $\varkappa \|\lambda_\varkappa - L\bulet \nu_\varkappa\|_{H^{-1}(\Omega; \R^{d+1})}^2$ is bounded independently of $\varkappa$, necessarily $\lambda = L \bulet \nu$ holds for the weak limits.

Then thanks to the weak-lower semi-continuity we have that
\begin{align*}
\int_\O &\Big(\phi\bulet \nu {+} (\theta{{-}}\theta_{\rm c}) \ccoupl{\cdot}\lambda{-}\psi_0(\theta){+}\frac{1}{2}m{\cdot}\nabla u_m{-} h(t){\cdot}m\big)\,\d x \\&\leq \liminf_{\varkappa \to \infty} \int_\O \Big(\phi\bulet \nu_\varkappa{+} (\theta{{-}}\theta_{\rm c}) \ccoupl{\cdot}\lambda_\varkappa{-}\psi_0(\theta){+}\frac{1}{2}m_\varkappa{\cdot}\nabla u_{m_\varkappa}{{-}} h(t){\cdot}m_\varkappa \big)\,\d x \\& \leq \liminf_{\varkappa \to \infty} \int_\O \Big(\phi\bulet \nu_\varkappa{+} (\theta{{-}}\theta_{\rm c}) \ccoupl{\cdot}\lambda_\varkappa{-}\psi_0(\theta){+}\frac{1}{2}m_\varkappa{\cdot}\nabla u_{m_\varkappa}{-} h(t){\cdot}m_\varkappa \Big)\,\d x {{+}} \frac \varkappa2 \|\lambda_\varkappa {-} L\bulet \nu_\varkappa\|_{H^{{-}1}(\Omega; \R^{d{+}1})}^2 \\
&\leq \liminf_{\varkappa \to \infty}  \int_\O \Big(\phi\bulet \hat{\nu}{+} (\theta{{-}}\theta_{\rm c}) \ccoupl{\cdot}\hat{\lambda}{-}\psi_0(\theta){+}\frac{1}{2}\hat{m}\nabla u_{\hat{m}}{-} h(t){\cdot}\hat{m}\Big)\,\d x {+} \frac \varkappa2 \|\hat{\lambda} {-} L\bulet \hat{\nu}\|_{H^{{-}1}(\Omega; \R^{d{+}1})}^2 \\ &= \int_\O \Big(\phi\bulet \hat{\nu}{+} (\theta{{-}}\theta_{\rm c}) \ccoupl{\cdot}\hat{\lambda}{-}\psi_0(\theta){+}\frac{1}{2}\hat{m}\nabla u_{\hat{m}}{-} h(t){\cdot}\hat{m}\Big)\,\d x
\end{align*}
for any $(\hat{\lambda}, \hat{\nu})$ such that $\hat{\lambda} = L \bulet \hat{\nu}$, which shows that $(\lambda, \nu)$ is a solution to \eqref{PenalizedStaticConverges}.

\noindent \emph{The rate-independent isothermal case:}

When considering the rate-independent case we formally set $\epsilon =0$ in \eqref{form-of-zeta}, i.e. assume that the dissipation potential is equal to $\delta^*_S$. Also, since now the dissipation potential yields less regularity on $\lambda$ we have to alter the specific mesoscopic Gibbs free energy and to add a regularization term $\gamma |\nabla \lambda|^2$; let us therefore denote
$$
g_{\text{\tiny RI}}(t,\nu,\lambda) = g(t, \nu, \lambda, \theta) + \gamma |\nabla \lambda|^2,
$$
again for some $\theta$ fixed. Now, as the temperature is not a variable in this context, the system of governing equations \eqref{ContinuousSystem} reduces to (with $g_{\text{\tiny RI}}(t,\nu,\lambda)$ replacing $ g(t, \nu, \lambda, \theta)$)
\begin{subequations} \label{ContinuousSystem-RI}
\begin{align}
&\!\!\!\!\!\!
\left.\begin{array}{llr}
\mbox{minimize} &
\displaystyle{
\int_\O\!\!\Big(g_{\text{\tiny RI}}(t,\nu,\lambda,)
+\frac \varkappa 2\big|\nabla\Delta^{-1}(\lambda(t){-}L\bulet\nu)\big|^2
\Big)\,\d x}\!\!\!\!\!\!
\\[2mm]
\mbox{subject to}\!& m=\id \bulet\nu,\ \ \ \ 
\hspace{8.0em}\mbox{ on } \O,\\[2mm]
&
\displaystyle{\mbox{div} \big(\mu_0\nabla u_m-\chi_\O m \big)=0}
\hspace{3.4em}\mbox{ on } \R^d,\\[2mm] &
\nu\!\in\! {\mathscr Y}^p(\O;\R^d),\ m\!\in\! L^p(\O;\R^d),
\ u_m\!\in\! H^1(\R^d),
\end{array}\hspace*{.5em}\right\}
\text{for }t\!\in\![0,T],
\hspace*{-1em}
\\
&\partial\delta_S^*(\DT{\lambda})+(\theta{-}\theta_{\rm c})\ccoupl + 2\gamma \diver \nabla \lambda + 
\varkappa\Delta^{-1}({\rm div}\,(\lambda-L\bulet \nu)) \ni 0
\hspace{2em}\text{in $Q:=[0,T]{\times}\Omega$,} \
\end{align}
\end{subequations}
and we recover, apart from the penalization, a similar model as in \cite{kruzik-roubicek,kruzik-roubicek-2}. Since we are considering the rate-independent case, a suitable weak formulation of \eqref{ContinuousSystem-RI}, which in fact -  due to the convexity of the problem - is equivalent to the standard weak formulation, is the so-called energetic formulation, cf. e.g. \cite{mielke-theil-levitas}. Then we shall call $(\lambda_\varkappa, \nu_\varkappa) \in L^\infty([0,T]; W^{1,2}(\Omega;\R^{d+1})) \cap \text{BV}([0,T]; L^1(\Omega; \R^{d+1})) \times ({\mathscr Y}^p(\O;\R^d))^{[0,T]}$ an energetic solution of \eqref{ContinuousSystem-RI} if they satisfy (we included initial conditions here already)
\begin{subequations}
\begin{align}
&\int_\Omega \Big( g_{\text{\tiny RI}}(t,\nu_\varkappa(t),\lambda_\varkappa(t)) + \frac \varkappa 2\big|\nabla\Delta^{-1}(\lambda_\varkappa(t){-}L\bulet\nu_\varkappa(t))\big|^2 \Big)  \dd x \nonumber \\ &\qquad \qquad   \leq \int_\Omega \Big(g_{\text{\tiny RI}}(t,\hat{\nu},\hat{\lambda}) + \frac \varkappa 2\big|\nabla\Delta^{-1}(\hat{\lambda}{-}L\bulet\hat{\nu})\big|^2  + \delta_S^*(\hat{\lambda} - \lambda_\varkappa) \Big) \dd x \nonumber \\ & \qquad \qquad \qquad \text{for all $(\hat{\lambda}, \hat{\nu}) \in W^{1,2}(\Omega;\R^{d+1}) \times {\mathscr Y}^p(\O;\R^d)$ and all $t \in [0,T]$ ,} \label{semistab} \\
& \int_\Omega  \Big(g_{\text{\tiny RI}}(T,\nu_\varkappa(T),\lambda_\varkappa(T)) + \frac \varkappa 2\big|\nabla\Delta^{-1}(\lambda_\varkappa(T){-}L\bulet\nu_\varkappa(T))\big|^2 \Big) \dd x   + \mathrm{Var}_{\delta_{S}^*}(\lambda; 0, T)  \nonumber \\ & \qquad \qquad \leq \int_\Omega g_{\text{\tiny RI}}(0,\nu_0,\lambda_0) \dd  x + \int_0^T \! \int_\Omega [g_{\text{\tiny RI}}]'_t(s,\nu_\varkappa(s)) \dd x \dd s, \label{energy-pen}\\
& \lambda_\varkappa(0) = \lambda_0 \in W^{1,2}(\Omega;\R^{d+1}), \ \ \ \nu_\varkappa(0) = \nu_0 \in {\mathscr Y}^p(\O;\R^d), \ \ \ \lambda_0 = L \bulet \nu_0, 
\end{align}
\end{subequations}
with $\mathrm{Var}_{f}(x; 0, T)$ the space integral of the variation of $f$ between $0$ and $T$. Let us now show that energetic solutions $(\lambda_\varkappa, \nu_\varkappa)$ of \eqref{ContinuousSystem-RI}  converge in $L^\infty([0,T]; W^{1,2}(\Omega;\R^{d+1})) \cap \text{BV}([0,T]; L^1(\Omega; \R^{d+1})) \times ({\mathscr Y}^p(\O;\R^d))^{[0,T]}$ (at least in terms of a subsequence) for $\varkappa \to \infty$ to $(\lambda, \nu)$ satisfying
\begin{subequations}
\begin{align}
&\int_\Omega g_{\text{\tiny RI}}(t,\nu(t),\lambda(t)) \dd x  \leq \int_\Omega \big(g_{\text{\tiny RI}}(t,\hat{\nu},\hat{\lambda}) + \delta_S^*(\hat{\lambda} - \lambda) \big) \dd x \nonumber \\ &\text{for all $(\hat{\lambda}, \hat{\nu}) \in W^{1,2}(\Omega;\R^{d+1}) \times {\mathscr Y}^p(\O;\R^d)$ such that $\hat{\lambda} = L \bulet \hat{\nu}$ and all $t \in [0,T]$ ,} \label{semistab-lim} \\
& \int_\Omega g_{\text{\tiny RI}}(T,\nu(T),\lambda(T)) \dd x + \mathrm{Var}_{\delta_{S}^*}(\lambda; 0, T)   \leq \int_\Omega g_{\text{\tiny RI}}(0,\nu_0,\lambda_0) \dd  x + \int_Q [g_{\text{\tiny RI}}]'_t(s,\nu(s)) \dd x \dd s,  \label{energy-lim} \\
& \lambda(0) = \lambda_0 \in W^{1,2}(\Omega;\R^{d+1}), \ \ \ \nu(0) = \nu_0 \in {\mathscr Y}^p(\O;\R^d), \ \ \ \lambda_0 = L \bulet \nu_0,
\end{align}
\end{subequations}
i.e. the energetic formulation of
\begin{subequations} \label{ContinuousSystem-RI-limited}
\begin{align}
&\!\!\!\!\!\!
\left.\begin{array}{llr}
\mbox{minimize} &
\displaystyle{
\int_\O\!\!g_{\text{\tiny RI}}(t,\nu,\lambda)\,\d x}\!\!\!\!\!\!
\\[2mm]
\mbox{subject to}\!& m=\id \bulet\nu, \ \ \ \ \lambda = L \bulet \nu, 
\hspace{5.9em}\mbox{ on } \O,\\[2mm]
&
\displaystyle{\mbox{div} \big(\mu_0\nabla u_m-\chi_\O m \big)=0}
\hspace{5.6em}\mbox{ on } \R^d,\\[2mm] &
\nu\!\in\! {\mathscr Y}^p(\O;\R^d),\ m\!\in\! L^p(\O;\R^d),
\ u_m\!\in\! H^1(\R^d),
\end{array}\hspace*{.5em}\right\}
\text{for }t\!\in\![0,T],
\hspace*{-1em}
\\
&\partial\delta_S^*(\DT{\lambda})+(\theta{-}\theta_{\rm c})\ccoupl + 2\gamma \div \nabla \lambda + \partial I_{\lambda = L \bulet \nu} \ni 0
\hspace{3.2em}\text{in $Q:=[0,T]{\times}\Omega$,} \
\end{align}
\end{subequations}
where $I_{\lambda = L \bulet \nu} = 0$, if $\lambda = L \bulet \nu$, and $+\infty$ otherwise.

The conjecture follows from the abstract paper \cite{mielke-roubicek-stefanelli}; here we give a very short sketch.

By similar tests as in the static case, it can be seen that $\|\lambda_\varkappa\|_{L^\infty([0,T]; W^{1,2}(\Omega;\R^{d+1})) \cap \text{BV}([0,T]; L^1(\Omega; \R^{d+1}))}$, $\int_\Omega |{\cdot}|^p \bulet \nu_\varkappa \dd x $ and $\varkappa \|\lambda_\varkappa - L \bulet \nu_\varkappa\|_{H^{-1}(\Omega; \R^{d+1})}^2$ are bounded independently of $\varkappa$.

Hence, by a slight modification of Helly's theorem \cite{mielke-diff, mielke-roubicek-stefanelli}, there exists $\lambda \in L^\infty([0,T]; W^{1,2}(\Omega;\R^{d+1})) \cap \text{BV}([0,T]; L^1(\Omega; \R^{d+1})$ and a not-relabeled subsequence of $\varkappa$ such that $\lambda_\varkappa(t) \to \lambda(t)$ for all $t \in [0,T]$ weakly in $W^{1,2}(\Omega;\R^{d+1})$ and hence strongly in $L^1(\Omega;\R^{d+1}))$. 

Let us now fix $t \in [0,T]$. Then there exists a subsequence of $\varkappa$ (dependent on $t$) denoted $\varkappa_t$ and a $\nu \in {\mathscr Y}^p(\O;\R^d))$ such that $\nu_{\varkappa_t} \to \nu$ weakly* in ${\mathscr Y}^p(\O;\R^d)$. Similarly as in the static case, using \eqref{semistab} and weak lower semi-continuity, we get that
\begin{align*}
\int_\Omega g_{\text{\tiny RI}}(t,\nu(t),\lambda(t)) \dd x \leq \liminf_{\varkappa_t \to \infty} \int_\Omega g_{\text{\tiny RI}}(t,\nu_{\varkappa_t}(t),\lambda_{\varkappa_t}(t)) \dd x + \frac {\varkappa_t}{2} \|\lambda_{\varkappa_t}-L \bulet \nu_{\varkappa_t}\|_{H^{-1}(\Omega; \R^{d+1})}^2 \\ \leq \liminf_{\varkappa_t \to \infty} \int_\Omega \Big(g_{\text{\tiny RI}}(t,\hat{\nu},\hat{\lambda}) + \delta_S^*(\hat{\lambda} - \lambda_{\varkappa(t)}(t)) \Big)\dd x = \int_\Omega \Big(g_{\text{\tiny RI}}(t,\hat{\nu},\hat{\lambda}) + \delta_S^*(\hat{\lambda} - \lambda(t))\Big)\dd x,
\end{align*}
for all $(\hat{\lambda}, \hat{\nu}) \in W^{1,2}(\Omega;\R^{d+1}) \times {\mathscr Y}^p(\O;\R^d)$ such that $\hat{\lambda} = L \bulet \hat{\nu}$, i.e. we showed that $(\lambda, \nu)$ fulfills \eqref{semistab-lim}. In the last line we exploited that $\delta_S^*$ is one-homogeneous and the \emph{strong convergence} $\lambda_\varkappa(t) \to \lambda(t)$ in $L^1(\Omega;\R^{d+1}))$.

To see that $(\lambda, \nu)$ also fulfills $\eqref{energy-lim}$ we pass to the limit in \eqref{energy-pen} exploiting only weak lower semi-continuity.

\noindent \emph{The thermally coupled case as exposed in \eqref{ContinuousSystem1}:}

As already mentioned, we only give a short heuristic sketch why the penalty approach is also justified in the case presented here, in particular we concentrate only on the limit passage in the minimization principle \eqref{ContinuousSystem1+} since this seems to be the most involved one. Let us for simplicity assume that $q = 2$.

Assume that $(\lambda_\varkappa, \nu_\varkappa) \in W^{1,2}([0,T]; L^2(\Omega; \R^{d+1})) \times ({\mathscr Y}^p(\O;\R^d))^{[0,T]}$ (together with some $w_\varkappa \in L^1([0,T]; W^{1,1}(\Omega))$, which is however irrelevant here) are weak solutions of \eqref{ContinuousSystem1} with an initial condition satisfying $\lambda_0 = L \bulet \nu_0$. Then \eqref{ContinuousSystem1+} yields, just by the chain rule, that
\begin{align*}
\int_\Omega \Big(\phi\bulet \nu_\varkappa(t) +&\frac{1}{2}
m_\varkappa(t){{\cdot}}\nabla u_{m_\varkappa}(t)- h(t){{\cdot}}m_\varkappa(t)
+\frac\varkappa2\big|\nabla\Delta^{-1}(\lambda_\varkappa(t){-}L\bulet\nu_\varkappa)\big|^2 \Big) \dd x\\ &= \int_\Omega  \Big( \phi\bulet \nu_0 +\frac{1}{2}
m_0(t){{\cdot}}\nabla u_{m_0}(t)- h(t){{\cdot}}m_0(t) \Big) \dd x \\& \qquad + \int_0^t \!\int_\Omega \Big(\varkappa \big(\nabla\Delta^{-1}(\lambda_\varkappa(s){-}L\bulet\nu_\varkappa(s))\big) {\cdot} \big(\nabla\Delta^{-1} \DT{\lambda}_\varkappa(s)\big) - \DT{h}(s){\cdot} m(s) \Big) \dd x \dd s 
\end{align*}
for any $t \in [0,T]$. Combining this with the flow rule tested by $\DT{\lambda}_\varkappa$ gives that $\sup_{t \in [0,T]} \int_\Omega |\cdot|^p \bulet \nu_\varkappa$ is bounded independently of $\varkappa $ and, moreover, that $\lambda_\varkappa$ is bounded independently of $\varkappa$ in $ W^{1,2}([0,T]; L^2(\Omega; \R^{d+1}))$, too. Using the estimates for $\lambda_\varkappa$ once again in the flow-rule, we get that $\int_0^T |\varkappa \int_\Omega \nabla \Delta^{-1}(\lambda_\varkappa - L \bulet \nu_\varkappa) \nabla \Delta^{-1} v \dd x | \dd t$ is bounded for all $v \in L^2(Q; \R^{d+1})$ such that $\|v\|_{L^2(Q; \R^{d+1})} \leq 1$, in particular $\{|\varkappa \int_\Omega \nabla \Delta^{-1}( \lambda_\varkappa(t) - L \bulet \nu_\varkappa(t)) \nabla \Delta^{-1} v_{\tiny \mathrm{S}} \dd x |\}_{\varkappa > 0}$ is bounded for all $v_{\tiny \mathrm{S}} \in L^2(\Omega; \R^{d+1})$ such that $\|v_{\tiny \mathrm{S}}\|_{L^2(\Omega; \R^{d+1})} \leq 1$ and a.a. $t \in [0,T]$. This in turn means that \{$\varkappa \|\lambda_\varkappa(t) - L \bulet \nu_\varkappa(t)\|_{H^{-1}(\Omega; \R^{d+1})}\}_{\varkappa > 0}$ is bounded for almost all $t \in [0,T]$.  

Using that $W^{1,2}([0,T]; L^2(\Omega; \R^{d+1})) \subset C([0,T]; L^2(\Omega; \R^{d+1})$ one can select a subsequence of $\varkappa$ (not relabeled) and find $\lambda \in W^{1,2}([0,T]; L^2(\Omega; \R^{d+1}))$, such that $\lambda_\varkappa(t) \to \lambda(t)$ strongly in $H^{-1}(\Omega; \R^{d+1})$ for all $t \in [0,T]$. 

\emph{Let us fix some $t \in [0,T]$ such that $\{\varkappa \|\lambda_\varkappa(t) - L \bulet \nu_\varkappa(t)\|_{H^{-1}(\Omega; \R^{d+1})}\}_{\varkappa > 0}$ is bounded (note that this is possible a.a. $t \in \Omega$}). Then similarly as in the isothermal rate-independent case one can find a subsequence of $\varkappa$ (dependent on $t$) denoted $\varkappa_t$ and $\nu \in {\mathscr Y}^p(\O;\R^d)$ such that $\nu_{\varkappa_t} \to \nu$ weakly* in ${\mathscr Y}^p(\O;\R^d)$. Once again by lower semi-continuity one can get the minimization principle 
\begin{align}
\int_\O &\Big(\phi\bulet \nu \frac{1}{2}m{\cdot}\nabla u_m{-} h(t){\cdot}m \Big)\,\d x \leq \liminf_{{\varkappa_t} \to \infty} \int_\O \Big(\phi\bulet \nu_{\varkappa_t}{+} \frac{1}{2}m_{\varkappa_t}{\cdot}\nabla u_{m_{\varkappa_t}}{{-}} h(t){\cdot}m_{\varkappa_t})\,\d x \nonumber \\& \leq \liminf_{{\varkappa_t} \to \infty} \int_\O \Big(\phi\bulet \nu_{\varkappa_t} {+}\frac{1}{2}m_{\varkappa_t}{\cdot}\nabla u_{m_{\varkappa_t}}{-} h(t){\cdot}m_{\varkappa_t})\,\d x {{+}} \frac{{\varkappa_t}}{2} \|\lambda_{\varkappa_t} {-} L\bulet \nu_{\varkappa_t}\|_{H^{{-}1}(\Omega; \R^{d{+}1})}^2 \nonumber \\
&\leq \liminf_{{\varkappa_t} \to \infty}  \int_\O \Big(\phi\bulet \hat{\nu}{+} \frac{1}{2}\hat{m}\nabla u_{\hat{m}}{-} h(t){\cdot}\hat{m} \big)\,\d x {+} \frac{{\varkappa_t}}{2} \|\lambda_{\varkappa_t}(t) {-} L\bulet \hat{\nu}\|_{H^{{-}1}(\Omega; \R^{d{+}1})}^2 
\label{Conv-Full-PenI}
\end{align}
for any $\hat{\nu} \in {\mathscr Y}^p(\O;\R^d)$. It would seem logical to take $\hat{\nu}$ such that $\lambda(t) = L \bullet \hat{\nu}$, then the penalization term on the right-hand side would become $\varkappa\|\lambda_\varkappa(t) - \lambda(t)\|^2_{H^{-1}(\Omega; \R^{d+1})}$. However, although we know that $\|\lambda_\varkappa(t) - \lambda(t)\|^2_{H^{-1}(\Omega; \R^{d+1})}$ converges to $0$ as $\varkappa \to \infty$ this no longer needs to hold if the term is multiplied by $\varkappa$. 

The limit passage in the minimization principle in the thermally coupled rate-dependent case therefore seems to be much more involved than in the cases presented before. This is mainly due to the fact that now the evolution of $\lambda$ is given by a completely separate equation. However, we can use a trick to circumpass this problem. Namely we realize that, \emph{due to the convexity}, any solution of \eqref{ContinuousSystem1+} $\nu_\varkappa$, solves also the following problem (for the fixed $t$)
\begin{equation}
\!\!\!\!\!\!
\begin{array}{llr}
\mbox{minimize} &
\displaystyle{
\int_\O\!\! \Big(\phi\bulet \nu +\frac{1}{2}
m{\cdot}\nabla u_m}- h(t){\cdot}m
+\varkappa(\nabla\Delta^{-1}(\lambda_\varkappa(t){-}L\bulet\nu_\varkappa))\cdot(\nabla\Delta^{-1}(\lambda_\varkappa(t){-}L\bulet\nu))
\Big) \,\d x\!\!\!\!\!\!
\\[2mm]
\mbox{subject to}\!& m=\id \bulet\nu,\ \ \ \ h_{\rm dem}=\nabla u_{m} \ \ \ \ \nu_\varkappa \text{ solution to \eqref{ContinuousSystem1+}}
\hspace{2.4em}\mbox{ on } \O,\\[2mm]
&
\displaystyle{\mbox{div} \big(\mu_0\nabla u_m-\chi_\O m \big)=0}
\hspace{3.4em}\mbox{ on } \R^d,\\[2mm] &
\nu\!\in\! {\mathscr Y}^p(\O;\R^d),\ m\!\in\! L^p(\O;\R^d),
\ u_m\!\in\! H^1(\R^d).
\end{array}
\label{LinProblem}
\end{equation}

Now we return to the second line \eqref{Conv-Full-PenI} (and use that ${\varkappa_t}/2 \leq {\varkappa_t}$) but instead of exploiting that $\nu_{\varkappa_t}$ solves \eqref{ContinuousSystem1+} we use that it solves \eqref{LinProblem} and get
\begin{align*}
 &\liminf_{{\varkappa_t} \to \infty} \int_\O \Big(\phi\bulet \nu_{\varkappa_t} {+}\frac{1}{2}m_{\varkappa_t}{\cdot}\nabla u_{m_{\varkappa_t}}{-} h(t){\cdot}m_{\varkappa_t})\,\d x {{+}}  {\varkappa_t} \|\lambda_{\varkappa_t} {-} L\bulet \nu_{\varkappa_t}\|_{H^{{-}1}(\Omega; \R^{d{+}1})}^2 \nonumber  \\ &= 
\liminf_{{\varkappa_t} \to \infty} \int_\O \Big(\phi\bulet \nu_{\varkappa_t} {+}\frac{1}{2}m_{\varkappa_t}{\cdot}\nabla u_{m_{\varkappa_t}}{-} h(t){\cdot}m_{\varkappa_t} {+} {\varkappa_t} \nabla \Delta^{-1}( \lambda_{\varkappa_t} {-} L\bulet \nu_{\varkappa_t})\nabla \Delta^{-1}( \lambda_{\varkappa_t} {-} L\bulet \nu_{\varkappa_t}) \Big) \dd x \nonumber \\
&\leq \liminf_{{\varkappa_t} \to \infty}  \int_\O \Big(\phi\bulet \hat{\nu}{+} \frac{1}{2}\hat{m}\nabla u_{\hat{m}}{-} h(t){\cdot}\hat{m} {+} {\varkappa_t}
\nabla \Delta^{-1}(\lambda_{\varkappa_t} {-} L\bulet \nu_{\varkappa_t})\nabla \Delta^{-1}( \lambda_{\varkappa_t} {-} L\bulet \hat{\nu}) \Big) \dd x 
\end{align*}
for all $\hat{\nu} \in {\mathscr Y}^p(\O;\R^d)$ such that $\lambda = L \bulet \hat{\nu}$. Now we exploit our special choice of $t$ for which we know that ${\varkappa_t} \|\lambda_{\varkappa_t}(t) - L \bulet \nu_{\varkappa_t}(t)\|_{H^{-1}(\Omega; \R^{d+1})}$ is bounded and $\lambda_{\varkappa_t}(t) \to \lambda(t)$ strongly in $H^{-1}(\Omega; \R^{d+1})$, which makes the penalization term vanish. Hence we establish the converged minimization principle.
}
\section{Weak formulation, qualification and main results}
In this section we shall give a weak formulation of the proposed micromagnetics model. The used formulation is to a great extend inspired by the \emph{energetic formulation} for rate-independent processes (see e.g.\ \cite{mielke-theil-levitas}) and its generalization given in e.g.\ \cite{tr1} for problems that include both rate-independent and rate-dependent processes.

In our case as well, we may regard the magnetic variable $\nu$ to be fast
evolving; its evolution is therefore driven by rate-independent processes. On
the other hand, the variables $\theta$ and $\lambda$ are evolving slowly as
there evolution is driven by a rate-dependent process. Hence, we demand the fast
evolving variable to satisfy a minimization principle and for the slow
variables we just require standard weak formulation of (\ref{ContinuousSystem1},b-d).

\begin{definition}[{\sc Weak solution}]
\label{DefinOfSol}
The triple $(\nu,\lambda, w){\in} ({\mathscr Y}^p(\Omega;\R^d))^{[0,T]}{\times}
W^{1,q}([0,T]; L^q(\Omega;\R^{d+1}))$ ${\times} L^1([0,T]; W^{1,1}(\Omega))$
such that $m=\mathrm{id} \bulet \nu \in L^2(Q;\R^d)$ and $L\bulet\nu\in L^{2}(Q;\R^{d+1})$ is called a weak solution to \eqref{ContinuousSystem1}
if it satisfies:
\begin{enumerate}
\item The {\bf minimization principle}: For all $\tilde \nu$ in ${\mathscr Y}^{p}(\O;\R^d)$ and all $t \in [0,T]$
\begin{equation}
\hspace{-2em}
\Gm(t,\nu,\lambda,\mathcalI(w))\le\Gm(t,\tilde{\nu},\lambda,\mathcalI(w)).
\label{balanceLaw}
\end{equation}
\item The {\bf reduced Maxwell system for magnetostatics}:
For a.a. $t \in [0,T]$ and all $v \in H^1(\R^d)$
\begin{equation}
\hspace{-2em}
\mu_0 \int_{\R^d} \nabla u_{m} \cdot \nabla v\,\dd x=
\int_{\Omega} m \cdot \nabla v\,\dd x.
\label{maxwellSysWeak}
\end{equation}
\item The {\bf flow rule}: For any $v \in L^q(Q; \R^{d+1})$
\begin{align}\nonumber
\hspace{-2em}
\int_Q\!\!\ \left( \big(\mathcalI(w){-}\theta_{\rm c}\big)\ccoupl{\cdot}
\big(v{-}\DT{\lambda}\big)+\delta_S^*(v)+\frac{\epsilon}{q} |v|^q
+\varkappa\nabla\Delta^{-1}(\lambda{-}L\bulet\nu){\cdot}
\nabla\Delta^{-1}(v{-}\DT{\lambda})
\right) \,\dd x\dd t
\\ \geq \int_{Q} \Big(\delta_S^*(\DT{\lambda})+\frac\epsilon q|\DT{\lambda}|^q \Big)\,\dd x\dd t.
\label{FlowRule}
\end{align}
\item The {\bf enthalpy equation}: For any $\varphi \in C^1(\bar{Q}), \, \varphi(T)=0$
\begin{align}
\hspace{-4em}
\int_{Q}\Big(\mathcal{K}(\lambda,w)\nabla w{\cdot}\nabla\varphi-w \DT{\varphi} \Big)\,\dd x \dd t + \int_{\Sigma} b\mathcalI(w)\varphi\,\dd S \dd t \nonumber
= \int_\Omega
w_0 \varphi(0) \,\dd x \\ + \int_{Q}\!\Big( \delta_S^*(\DT{\lambda})+\epsilon|\DT{\lambda}|^q+ \mathcalI(w)\ccoupl{\cdot}\DT{\lambda}\Big)\varphi \,\dd x \dd t + \int_{\Sigma} b\theta_{\mathrm{ext}}\varphi \,\dd S \dd t.
\label{EnthalpyEq}
\end{align}
\item The {\bf remaining initial conditions}
in \eqref{initCond}: $\nu(0,\cdot)=\nu_0$ and $\lambda(0,\cdot)=\lambda_0$.
\end{enumerate}
\end{definition}

\noindent {\bf Data qualifications:}\\
Let us now summarize the data qualification, needed to prove the existence of
weak solutions:
\begin{subequations}\label{ass}
\begin{align}
\intertext{Isothermal part of the anisotropy energy: $\phi \in C(\R^d)$ and}
&\exists c^A_1,c^A_2 >0,\ p>4:\ \ c^A_1(1+|\cdot|^p) \leq \phi(\cdot) \leq c^A_2(1+|\cdot|^p), \label{pGrowth}
\intertext{Rate-independent dissipation: $\delta_S^* \in C(\R^{d+1})$ positively homogeneous, and}
&\exists c_{1,D},c_{2,D} >0:\ \ c_{1,D}(|\cdot|) \leq \delta_S^*(\cdot) \leq c_{2,D}(|\cdot|), \label{GrowthDissip}
\intertext{External magnetic field:}
&h \in C^1([0,T]; L^2(\Omega; \R^d)),\label{ass-f}
\intertext{Specific heat capacity: $c_\mathrm{v} \in C(\R)$ and,
with $q$ from \eqref{form-of-zeta},}
&\exists c_{1,\theta},c_{2,\theta}>0,\ \omega_1 \geq \omega \geq q',\
c_{1,\theta} (1{+}\theta)^{\omega-1} \leq c_\mathrm{v}(\theta) \leq c_{2,\theta}(1{+}\theta)^{\omega_1-1},
\label{OmegaGrowth}
\intertext{Heat conduction tensor: $\mathcal{K}\in C(\R^{d+1}\times\R;\R^{d \times d})$ and }
&\exists C_K,\kappa_0 >0\ \forall\chi\in\R^d: \mathcal{K}(\cdot , \cdot) \leq C_K, \, \, \, \, \chi^\mathrm{T}\mathcal{K}(\cdot, \cdot) \chi \geq \kappa_0 |\chi|^2,
\label{Sub6Growth}
\intertext{External temperature:}
&\theta_\mathrm{ext} \in L^1(\Sigma),\ \ \theta_\mathrm{ext}\ge0,
\ \text{ and }\ b \in L^\infty(\Sigma),\ \ b\ge0,
\label{ExternalTemp}
\intertext{Initial conditions:}
&\nu_0\in{\mathscr Y}^p(\Omega;\R^d) \text{ solving \eqref{balanceLaw} },\quad \lambda_0\in L^q(\O;\R^{d+1}),
\quad w_0=\widehat c_\mathrm{v}(\theta_0)\in L^1(\O)
\text{ with }\theta_0\ge0. \label{ass-IC}
\end{align}
\end{subequations}
Note that \eqref{IntroduceEnthalpy} combined with \eqref{OmegaGrowth} yields for non-negative $w$
\begin{equation}
w=\widehat c_\mathrm{v}(\theta)=\int_0^\theta\!\!c_\mathrm{v}(r) \dd r \geq c_{1,\theta} \int_0^\theta\!(1{+}r)^{\omega-1} \dd r \geq c_{1,\theta} \!\big( (1{+}\theta)^\omega\!-1\big) =c_{1,\theta}\big((1{+}\mathcalI(w))^\omega\!-1\big).
\label{OmegaGrowthBound}
\end{equation}

The main analytical result we will prove in the following sections is:

\begin{theorem}\label{thm-main}
Let \eqref{ass} hold. Then at least one weak solution $(\nu,\lambda, w)$
to the problem \eqref{ContinuousSystem1}
in accord with Definition~\ref{DefinOfSol} does exist. Moreover, some of
these solutions satisfies also
\begin{align}\label{estimates-of-w}
&w\in L^r([0,T]; W^{1,r}(\Omega))\,\cap\,W^{1,1}(I;W^{1,\infty}(\Omega)^*)
\qquad\text{ with }1\le r<\frac{d{+}2}{d{+}1}.
\end{align}
\end{theorem}

\begin{proof}[Scenario of the proof]
The above assertion will immediately follow from the 
Proposition~\ref{MainConvProp}, which assures convergence
of the approximate solutions constructed by semi-implicit time discretization in Section~\ref{sect-disc} to the weak solutions of \eqref{ContinuousSystem1},   when proving a-priori estimates in Proposition~\ref{aprioriII} and when realizing that the approximations $w_{0,\tau}$ and
$\lambda_{0,\tau}$ of $w_0$ and $\lambda_0$ required in \eqref{lambdaZeroApprox} below always exist.
The information $\frac{\partial w}{\partial t}\in L^1(I;W^{1,\infty}(\Omega)^*)$ can be obtained from the equation (\ref{ContinuousSystem1}c,d) itself.
\end{proof}

\section{Time discretization, a-priori estimates, and convergence}\label{sect-disc}
To prove Theorem~\ref{thm-main}, we proceed in a \emph{constructive manner} that may serve also as a \emph{conceptual numerical algorithm}, at least after a spatial discretization being performed. Namely, we discretize the time with a time-step $\tau$ and introduce a minimization problem in every time-step that we shall call the \emph{time-incremental minimization problem}. This problem represents a discrete version of the minimization principle as well as of the flow rule. Also, we apply a semi-implicit method of time-discretization in such a way that it decouples the time incremental minimization problem and the heat equation in any particular time-step $k$ by using the ``retarded''
enthalpy, i.e.\ $w_\tau^{k-1}$.

We call the triple $(\nu_\tau^k, \lambda_\tau ^k, w_\tau ^k) \in {\mathscr Y}^{p}(\O;\R^d) \times L^{2q}(\Omega; \R^{d+1})\times H^1(\Omega)$ the \emph{discrete weak solution} of \eqref{ContinuousSystem1} subject to boundary condition \eqref{boundaryCondEnthalp} at time-level $k$, $k = 1 \ldots T/\tau$, if it satisfies:
\begin{subequations}\label{def-disc}
\begin{enumerate}
\item The time-incemental {\bf minimization problem} with given $\lambda_\tau ^{k-1}$ and $w_\tau ^{k-1}$:
\begin{align}
\left.
\hspace{-2ex}
\begin{array}{ll}
\mathrm{Minimize}\!\!& \displaystyle{\Gm(k\tau,\nu,\lambda,\mathcalI(w_\tau^{k-1}))
+\tau\!\int_\Omega\!\left(|\lambda|^{2q}
+\delta_S^*\Big(\frac{\lambda{-}\lambda_\tau ^{k-1}}{\tau}\Big)
+\frac{\epsilon}{q}\Big|\frac{\lambda{-}\lambda_\tau ^{k-1}}{\tau}\Big|^q \right) \dd x}
\\[.7em]
\mathrm{subject \, to}\!\!&
(\nu,\lambda)\in {\mathscr Y}^{p}(\O;\R^d) \times L^{2q}(\Omega; \R^{d+1}).
\end{array}\!\right\}\!\!\!
\label{BalanceEqDis}
\end{align}
with $\Gm$ from \eqref{def-of-psi}.
\item The {\bf reduced Maxwell system for magnetostatics}:
For all $v \in H^1(\R^d)$
\begin{equation}
\int_{\R^d} \nabla u_{m_\tau^k}{\cdot}\nabla v\,\dd x = \int_{\Omega}m_\tau^k {\cdot}\nabla v\,\dd x\qquad\text{ with }\ m_\tau^k=\mathrm{id} \bulet \nu_\tau^k.
\label{MaxwellDis}
\end{equation}
\item The {\bf enthalpy equation}: For all $\varphi \in H^1(\Omega)$
\begin{align}
\int_{\Omega}\left(\frac{w_\tau^k{-}w_\tau ^{k-1}}{\tau}\varphi + \mathcal{K}(\lambda_\tau^k, w_\tau^k) \nabla w_\tau ^k{\cdot}\nabla\varphi \right)\,\dd x
+ \int_{\Gamma} b^k_\tau\mathcalI(w_\tau ^k)\varphi\,\dd S \nonumber
= \int_{\Gamma} b^k_\tau\theta^k_{\mathrm{ext}, \tau}\varphi\,\dd S
\\
+\int_{\Omega}\left(\delta_S^*
\Big(\frac{\lambda_\tau^k{-}\lambda_\tau^{k-1}}{\tau}\Big)
+\epsilon\Big|\frac{\lambda_\tau^k{-}\lambda_\tau^{k-1}}{\tau}\Big|^q
\mathcalI(w_\tau^k)\ccoupl{\cdot}
\frac{\lambda^k{-}\lambda^{k-1}}{\tau} \right)\varphi\,\dd x.
\label{EnthalpyEqDis}
\end{align}
\item For $k=0$ the {\bf initial conditions} in the following sense
\begin{align}\label{initCondDis}
\nu_\tau ^0=\nu_0,\qquad \lambda_\tau ^0=\lambda_{0,\tau},
\qquad w_\tau ^0=w_{0,\tau}\ \ \ \text{ on }\O.
\end{align}
\end{enumerate}
\end{subequations}

In \eqref{initCondDis}, we denoted by $\lambda_{0,\tau}\in L^{2q}(\O;\R^{d+1})$
and $w_{0,\tau}\in L^2(\O)$
respectively suitable approximation of the
original initial conditions $\lambda_0\in L^q(\O;\R^{d+1})$ and
$w_0\in L^1(\O)$ such that
\begin{subequations}
\begin{align}\label{lambdaZeroApprox}
&\lambda_{0,\tau}\to\lambda_0\ \text{ strongly in $L^q(\O;\R^{d+1})$, and }\
\|\lambda_{0,\tau}\|_{L^{2q}(\O;\R^{d+1})} \leq C \tau^{-1/(2q+1)},
\\
&w_{0,\tau}\to w_0\ \text{ strongly in $L^1(\O)$, and }\ w_{0,\tau}\in L^2(\O).
\end{align}
\end{subequations}
Moreover $\theta^k_{\mathrm{ext}, \tau} \in L^2(\Gamma)$ and $b^k_\tau \in L^\infty (\Gamma)$ are defined in such a way that their piecewise constant interpolants
\begin{equation*}
\big[\bar\theta_{\mathrm{ext},\tau},\bar b_\tau](t):=\big(\theta^k_{\mathrm{ext}, \tau},b^k_\tau,)
\qquad\qquad\text{ for $\ (k{-}1)\tau<t\le k\tau$,\ \ $k=1,...,K_\tau$}.
\end{equation*}
satisfy
\begin{equation}
\bar\theta_{\mathrm{ext}, \tau} \to \theta_\mathrm{ext}\ \text{ strongly in $L^1(\Sigma)$ and }\ \bar b_\tau \stackrel{*}{\rightharpoonup} b \ \text{ weakly* in $L^\infty(\Sigma)$.}
\label{ConvTempData}
\end{equation}

To see why \eqref{def-disc}
indeed forms a correct time-discretization of the weak formulation of \eqref{DefinOfSol}, note that \eqref{BalanceEqDis} already contains the flow rule since, if we can find a minimizer, the first order optimality condition in $\lambda$ evaluated at $\lambda_\tau ^k$ yields the discrete version of \eqref{FlowRule}.

Realize also, that we have added the regularization term
$\tau |\lambda|^{2q}$ to the time-incremental problem, cf.\
\eqref{BalanceEqDis}. This shall assure
that even
$|\frac{\lambda^k{-}\lambda^{k-1}}{\tau}|^q \in L^2(\Omega)$ (although
it does not hold uniformly for $\tau\to 0$)
and in turn also the existence of solutions of the enthalpy equation
in the classical weak sense.
Of course, this regularization term will be shown to
converge to 0 as we pass to the limit $\tau \to 0$.

\begin{proposition}[Existence of discrete solutions]
\label{ExistenceDiscSol} Let \eqref{ass} hold and let also $w_\tau ^0\geq 0$.
Then there exists a discrete weak solution according to \eqref{def-disc}
such that $w_\tau ^k \geq 0$ for all $k=1 \ldots T/\tau$.
\end{proposition}

\begin{proof}
First note that \eqref{EnthalpyEqDis} is decoupled from
\eqref{BalanceEqDis} and \eqref{MaxwellDis}.
It is easy to
see that,
for any $\nu \in {\mathscr Y}^{p}(\O;\R^d)$, there exists a unique
$u_{m} \in H^1(\R^d)$ that solves \eqref{MaxwellDis} since $m \equiv \id \bulet \nu \in L^2(\Omega; \R^{d})$. Therefore we may proceed by the \emph{direct method} to prove existence of \eqref{BalanceEqDis}, \eqref{MaxwellDis}, i.e. take at time-step $k$ a minimization sequence $\{q_{k,j}\}_{j=0}^\infty = \{(\nu_{k,j}, \lambda_{k,j})\}_{j=0}^\infty$ of \eqref{BalanceEqDis}. Due to the coercivity of the cost functional (thanks to assumption \eqref{pGrowth}) this minimizing sequence converges weakly*(at least in terms of a subsequence) in the space $L^\infty_\mathrm{w}(\Omega; \mathcal{M}(\R^d)) \times L^q(\Omega; \mathrm{R}^{d+1})$ to some $q_k$. Moreover, note that, again due to assumption \eqref{pGrowth}, $\nu_k$ is a Young measure. Then, by the convexity of the functional in $\lambda$, $\nu$, by the fact that $m_{k,j} \rightharpoonup m_k$ in $L^2(\Omega; \R^d)$ (because $\nu_{k,j}$ as well as $\nu_k$ are in ${\mathscr Y}^{p}(\O;\R^d)$ and $\nu_{k,j}\stackrel{*}{\rightharpoonup} \nu_k$ in $L^\infty_\mathrm{w}(\Omega ; \mathcal{M}(\R^d))$) and the boundedness from below of $\phi$, $q_k$ is the sought minimizer of \eqref{BalanceEqDis}
at time-step $k$.

The existence of solutions to \eqref{EnthalpyEqDis} for $k=1$ (and subsequently
also for all other $k$) can be proved by standard methods exploiting theory
of pseudomonotone operators (here rather trivially since the problem is
semi-linear and thus the underlying operator $H^1(\Omega)\to H^1(\Omega)^*$
is weakly continuous).
Note that the right-hand side can be
represented as an element of $H^1(\Omega)^*$ due to the integrability of the initial data, the suitable choice of the time-discretization of the
external heat flux $\theta^k_{\mathrm{ext}, \tau} \in L^2(\Gamma)$, and
thanks to the
regularization term $\tau |\lambda|^{2q}$ in \eqref{BalanceEqDis}
which causes $|\frac{\lambda^k{-}\lambda^{k-1}}{\tau}|^q \in L^2(\Omega)$,
as already mentioned above.

Let us test \eqref{EnthalpyEqDis} by $[w_\tau ^k]^-\equiv \min{(0,w_\tau^k)}$, which is an element of $ H^1(\Omega)$ as $w_\tau ^k \in H^1(\Omega)$ and hence a legal test function. We get
\begin{align*}
&\int_{\Omega} \Big( w_\tau ^k [w_\tau ^k]^- + \tau \mathcal{K}(\lambda^k_\tau, w^k_\tau) \nabla w_\tau ^k \cdot \nabla [w_\tau ^k]^- \Big) \dd x \leq \int_{\Omega} \left( \delta_S^*\Big( \frac{\lambda_\tau ^k{-}\lambda_\tau^{k-1}}{\tau}\Big) [w_\tau ^k]^- +\epsilon|\frac{\lambda_\tau ^k{-}\lambda_\tau^{k-1}}{\tau}\Big|^q [w_\tau ^k]^- \right.
\\&\qquad
+\left. \mathcalI(w_\tau)\ccoupl{\cdot}\frac{\lambda_\tau ^k-\lambda_\tau ^{k-1}}{\tau} [w_\tau ^k]^- + w_\tau ^{k-1} [w_\tau ^k]^- \right) \dd x
- \int_{\Gamma} \Big(b^k_\tau \mathcalI(w_\tau ^k)[w_\tau ^k]^- + b^k_\tau \theta^k_{\mathrm{ext},\tau}[w_\tau ^k]^- \Big)\dd S.
\end{align*}

Now $\mathcalI(w_\tau ^k)[w_\tau^k]^-=0$; here we realize that we have defined $\mathcalI(w)=0$ for $w\le0$. Further
we realize that
$\delta_S^*( \frac{\lambda_\tau^k{-}\lambda_\tau^{k-1}}{\tau})\ge0$ and
$\epsilon|\frac{\lambda_\tau^k{-}\lambda_\tau^{k-1}}{\tau}|^q\ge0$ which implies
that $\delta_S^*(\frac{\lambda_\tau^k{-}\lambda_\tau^{k-1}}{\tau})[w_\tau ^k]^-\le0$ and $\epsilon|\frac{\lambda_\tau^k{-}\lambda_\tau ^{k-1}}{\tau}\Big|^q[w_\tau ^k]^-\le0$.
Using these facts and
by exploiting that $b^k_\tau\theta_\mathrm{ext,\tau}[w_\tau ^k]^-\le0$ (a consequence of \eqref{ExternalTemp}), we get
\begin{equation*}
\int_{\Omega}\big|[w_\tau ^k]^-\big|^2
+ \tau \kappa_0 \big|\nabla [w_\tau ^k]^-\big|^2 \dd x
\leq \int_{\Omega}w_\tau ^{k-1} [w_\tau ^k]^- \dd x.
\end{equation*}
Using this equation recursively and when also taking into account that
$w_\tau ^0 \geq 0$ gives $w_\tau^k \geq 0$.
\end{proof}

Let us introduce the notion of \emph{piecewise affine} interpolants
$\lambda_\tau$ and $w_\tau$ defined by
\begin{align}\nonumber
\big[\lambda_\tau,w_\tau\big](t):=\frac{t-(k{-}1)\tau}\tau \big(\lambda_\tau^k,w_\tau^k\big) +\frac{k\tau-t}\tau \big(\lambda_\tau^{k-1},w_\tau^{k-1}\big)
\qquad\text{ for $\ t\in[(k{-}1)\tau,k\tau]\ $}
\end{align}
with $\ k=1,...,T/\tau$.
In addition define the backward \emph{piecewise constant interpolants}
$\bar{\nu}_\tau$, $\bar{\lambda}_\tau$, and $\bar{w}_\tau$ by
\begin{align}
\big[\bar{\nu}_\tau,\bar{\lambda}_\tau,\bar{w}_\tau\big](t):=\big(\nu_\tau^k,\lambda_\tau^k,w_\tau^k\big)
\qquad\qquad\text{ for $\ (k{-}1)\tau<t\le k\tau$,\ \ $k=1,...,T/\tau$}.
\end{align}
Eventually, we will also need the ``retarded'' enthalpy and magnetization piecewise constant interpolant $\underline{w}_\tau$, $\underline{m}_\tau$ defined by
\begin{align}\label{w-backward}
[\underline{w}_\tau(t), \underline{m}_\tau(t)]:= [w_\tau^{k-1}, \id \bulet \nu_\tau^{k-1}] \qquad\qquad \text{ for $\ (k{-}1)\tau<t\le k\tau$,\ \ $k=1,...,T/\tau$}.
\end{align}

Note that from now on \emph{we use $C$ as a generic constant}, which ma change from expression to expression, and do not specify its dependence on the problem parameters such as $\epsilon, q, p, |\Omega|$.

\begin{proposition}[A-priori estimates]
\label{aprioriII}
Let the assumptions \eqref{ass} hold and let $\tau < \tau_0$ for some $\tau_0 > 0$ fixed. Then the interpolants of discrete weak solutions satisfy
\begin{align}
&\!\!\sup_{t\in [0,T]} \mbox{$\int_\Omega$}|\cdot|^{p}\bulet\bar{\nu}_{\tau} \dd x
\leq C, \label{Nu-Apriori} \\
&\big\| \DT{ \lambda}_\tau \big\|_{L^q(Q;\R^{d+1})} \leq C,
\label{LambdaApriori} \\
&\big\|\bar{\lambda}_\tau\big\|_{L^\infty([0,T]; L^{2q}(\O;\R^{d+1}))}\le C\tau^{-1/2q},
\label{lambdaBlowUp}\\
&\big\| \bar{w}_\tau\big\|_{L^\infty([0,T]; L^1(\Omega))}\leq C,
\label{EnthalpyLinftyL1} \\
&\big\| \nabla\bar{w}_\tau\big\|_{L^r(Q; \R^d)) }\leq C_r\qquad\text{ with any }\ 1\le r<\frac{d{+}2}{d{+}1},
\label{EnthalpyLinftyL1+} \\
&\big\| \DT{ w}_\tau \big\|_{\mathcal{M}([0,T]; {W^{1,\infty}(\Omega)}^*)}\leq C.
\label{DualEnthalpyEst}
\end{align}
\end{proposition}

Let us emphasize that our strategy of proving \eqref{LambdaApriori} and
\eqref{EnthalpyLinftyL1} will slightly deviate from the standard approach based on
testing the flow rule \eqref{FlowRuleBasic} by $\DT{\lambda}$ together with the entropy equation \eqref{EntropyEquation} by 1, which, when adding these two, would lead to canceling of the dissipative heat rate as well as the calorimetric term $\mathcalI(w)\ccoupl{\cdot}\DT{\lambda}$, see e.g. \cite{tr1}.
In contrast to \cite{tr1} however, here we use the retarded enthalpy in the time-incremental minimization problem and hence the strategy of \cite{tr1} would not work in our case, as the calorimetric term \emph{would not cancel out.} Therefore we exploit the rate-dependent dissipation term that yields more regularity than the rate-independent contribution.

In what follows, we will use the abbreviation $\llangle\cdot,\cdot\rrangle$
for the scalar product in $H^{-1}(\O;\R^{d+1})$; in view of the specific
choice of the norm on $H^{-1}(\O;\R^{d+1})$ in Sect.~\ref{sect-evol}, this means
\begin{align}
\llangle\lambda,v\rrangle:=
\int_\O\nabla\Delta^{-1}\lambda{\cdot}\nabla\Delta^{-1}v\,\d x.
\end{align}
\begin{proof}[Proof of Proposition~\ref{aprioriII}] 
{
Before giving a rigorous proof we give a formal heuristic sketch, using the system \eqref{ContinuousSystem1}, on how estimates \eqref{Nu-Apriori}-\eqref{DualEnthalpyEst} can be established; we shall always point to the adequate step in this proof where the formal procedure is performed rigorously on the discrete level. For this heuristics only, we shall assume that all functions are as smooth as needed. Furthermore, let us denote by $(\nu, \lambda, \theta)$ the solutions of $\eqref{ContinuousSystem1}$.

First of all we exploit (\ref{ContinuousSystem1}a), which can be reformulated as minimizing part with respect to $\nu$ the magnetic  of the Gibbs free energy $\mathfrak{G}$ defined as
\begin{equation}\label{def-of-G-mag}
\mathfrak{G}(t,\nu, \lambda):=\int_\Omega \phi \bulet \nu -
h(t){\cdot}m\,\dd x + \int_{\R^d} \frac{1}{2}|\nabla u_{m}|^2\,\dd x
+\frac\varkappa2\big\|\lambda {-} L \bulet \nu\big\|_{H^{-1}(\Omega; \R^{d+1})}^2.
\end{equation}
Then just using the chain rule and realizing that the partial derivative of $\mathfrak{G}$ with respect to $\nu$ evaluated at the minimizer has to be zero leads to
\begin{align}
\mathfrak{G}(t_0,\nu(t_0), \lambda(t_0)) &= \mathfrak{G}(t_0,\nu(0), \lambda(0)) {+} \int_0^{t_0} \!\! \frac{\dd}{\dd t } \mathfrak{G}(t,\nu(t), \lambda(t)) \dd t \nonumber \\&= \mathfrak{G}(t_0,\nu(0), \lambda(0)) + \int_0^{t_0} \Big( \DT{h}(t) m {+} \varkappa \llangle\lambda{-}L\bulet \nu,\DT{\lambda}\rrangle \Big) \dd t,
\label{HeuristicBalanceLaw}
\end{align}
for somee arbitrary $t_0 \in [0,T]$. Using the Young inequality and assumption (\ref{ass}a) on the coercivity of the isothermal part of the anisotropy energy and also (\ref{ass}c) yields (cf. Step 1 below) the following estimate
\begin{equation}
C \int_\Omega |\cdot|^p \bulet \nu(t_0)
\leq \int_0^{t_0}\!\!\int_\Omega \Big(\frac{\e}{4q}|\DT{\lambda}|^q + C|\cdot|^{p} \bulet \nu
 \Big)\, \dd x \dd t + C.
\label{heuristics-1}
\end{equation}

Inequality \eqref{heuristics-1} needs to be combined with an estimate on $\DT{\lambda}$. In order to get it, we multiply the flow-rule (\ref{ContinuousSystem1+}b) by $\frac{1}{q} \DT{\lambda}$ and integrate over $\Omega$ and $[0,t_0]$, which yields (note that $\partial \delta_\mathrm{S}^*(\DT{\lambda}) \DT{\lambda} = \delta_\mathrm{S}^*(\DT{\lambda})$)
\begin{equation*}
\int_0^{t_0}\int_\Omega \frac{1}{q} \Big( \delta_\mathrm{S}^*(\DT{\lambda}) + \frac{\epsilon}{q} |\DT{\lambda}|^q \Big) \dd x \dd t  = \int_0^{t_0} \Big( \frac{\varkappa}{q} \llangle\lambda{-}L\bulet \nu,\DT{\lambda}\rrangle - \frac{1}{q} \int_\Omega \big(\mathcalI(w){-}\theta_{\rm c}\big)\ccoupl \cdot \DT{\lambda} \, \dd x \Big) \dd t.
\end{equation*}
Using the chain-rule, we can rewrite this as 
\begin{align}
\int_0^{t_0}&\int_\Omega \frac{1}{q} \Big( \delta_\mathrm{S}^*(\DT{\lambda}) + \frac{\epsilon}{q} |\DT{\lambda}|^q \Big) \dd x \dd t + \frac{\varkappa}{q}\|\DT{\lambda}(t_0)\|_{H^{-1}(\Omega; \R^{d+1})}  \nonumber\\ &\leq \frac{\varkappa}{q}\|\DT{\lambda}(0)\|_{H^{-1}(\Omega; \R^{d+1})} - \int_0^{t_0} \Big( \frac{\varkappa}{q} \llangle L\bulet \nu,\DT{\lambda}\rrangle + \frac{1}{q} \int_\Omega \big(\mathcalI(w){-}\theta_{\rm c}\big)\ccoupl \cdot \DT{\lambda} \, \dd x \Big) \dd t,
\label{flowRuleHeuristics}
\end{align}
which, by usage of the Young inequality (cf. Step 2 below), yields the following estimate 
\begin{equation}
\int_0^{t_0}\int_\Omega \frac{1}{q} \Big( \delta_\mathrm{S}^*(\DT{\lambda}) + \frac{\epsilon}{q} |\DT{\lambda}|^q \Big) \dd x \dd t + \frac{\varkappa}{q}\|\DT{\lambda}(t_0)\|_{H^{-1}(\Omega; \R^{d+1})} \leq \int_0^{t_0}\!\!\int_\Omega \Big(\frac{\e}{4q}|\DT{\lambda}|^q + C|\cdot|^{p} \bulet \nu + C|w|
 \Big)\, \dd x \dd t + C.
\label{heuristics-2}
\end{equation}
Finally, multiplying the enthalpy equation (\ref{ContinuousSystem1}c) by $\frac{\e}{8q}$, integrating over $\Omega$ and $[0,t_0]$, and using the Young inequality as already above gives
\begin{equation}
\frac{\e}{8q} \int_0^{t_0} \int_\Omega \DT{w} \dd x \dd t \leq  \int_0^{t_0} \int_\Omega \Big(\frac{\e}{4q} |\DT{\lambda}|^q + C|w| \Big) \dd x \dd t + C
\label{heuristics-3}
\end{equation}
Adding \eqref{heuristics-1}, \eqref{heuristics-2} and \eqref{heuristics-3} gives, on this heuristic level, \eqref{Nu-Apriori}, \eqref{LambdaApriori} and \eqref{EnthalpyLinftyL1}.

Estimate \eqref{EnthalpyLinftyL1+} is got by testing the enthalpy equation (\ref{ContinuousSystem1}c) by $1-\frac{1}{(1+w)^a}$ (cf. Step 4, below) while \eqref{DualEnthalpyEst} is got from the enthalpy equation itself.

For clarity, let us divide the formal part of the proof into five steps.
\hspace{1ex}

\noindent \emph{Step 1: Using the time-incremental minimization problem}:
In this step we perform the procedure that on the heuristic level led to \eqref{heuristics-1}; together with Step 2 and Step 3 it will give \eqref{Nu-Apriori}--\eqref{EnthalpyLinftyL1}.

As we know that $(\nu_\tau^l, \lambda_\tau^l)$ solves the time-incremental
problem \eqref{BalanceEqDis}, we may write
\begin{align*}
\Gm(l\tau, &\nu_\tau ^l, \lambda_\tau^l, \mathcalI(w_\tau^{l-1})) +
\int_\Omega\! \left(\tau |\lambda_\tau^l|^{2q}
+ \tau\delta_S^*\Big(\frac{\lambda_\tau^l{-}\lambda_\tau^{l-1}}{\tau}\Big)
+\frac{\tau\epsilon}q\Big|\frac{\lambda_\tau^l{-}\lambda_\tau^{l-1}}{\tau}\Big|^q
\right)\dd x
\\
\leq &\Gm(l\tau, \nu_\tau ^{l-1}, \lambda_\tau^l, \mathcalI(w_\tau^{l-1})) +
\int_\Omega\! \left(\tau |\lambda_\tau^l|^{2q}
+ \tau\delta_S^*\Big(\frac{\lambda_\tau^l{-}\lambda_\tau^{l-1}}{\tau}\Big)
+\frac{\tau\epsilon}q\Big|\frac{\lambda_\tau^l{-}\lambda_\tau^{l-1}}{\tau}\Big|^q
\right)\dd x,
\end{align*}
which can be rewritten using the magnetic part of the Gibbs free energy as
\begin{align*}
\mathfrak{G}(l\tau, \nu_\tau ^l, \lambda_\tau^l) &\leq \mathfrak{G}(l\tau, \nu_\tau ^{l-1}, \lambda_\tau^l) \\ &\leq \mathfrak{G}((l-1)\tau, \nu_\tau ^{l-1}, \lambda_\tau^{l-1}) + \int_{(l-1)\tau}^{l\tau} \Big(\mathfrak{G}'_t(l \tau, \nu_\tau ^l) +  \mathfrak{G}'_\lambda(\nu_\tau ^l, \lambda_\tau^l) \Big) \dd t 
\end{align*}
where the inequality on the second line is got by the discrete chain rule (relying on convexity). Realizing that $\mathfrak{G}'_t(l \tau, \nu_\tau ^l) = \int_\Omega \DT{h}_\tau(l\tau) \cdot m_\tau^l \dd x$ ($h_\tau$ denotes the piece-wise linear approximation of $h$) and $\mathfrak{G}'_\lambda(\nu_\tau ^l, \lambda_\tau^l) = \varkappa \llangle \lambda_\tau^l - L \bulet \nu_\tau ^l , \frac{\lambda_\tau^l-\lambda_\tau^{l-1}} {\tau}\rrangle$ and summing from $0$ to $k$ gives 
\begin{equation}
\mathfrak{G}(t_k, \bar{\nu}_\tau (t_k), \bar{\lambda}_\tau(t_k)) \leq  \mathfrak{G}(0, \bar{\nu}_\tau (0), \bar{\lambda}_\tau(0)) + \int_0^{t_k} \Big( \int_\Omega \DT{h}_\tau \cdot \bar{m}_\tau \dd x + \varkappa \llangle \bar{\lambda}_\tau - L \bulet \bar{\nu}_\tau, \DT{\lambda}_\tau\rrangle \Big) \dd t
\label{TestBalanceDisc}
\end{equation}
with $t_k = k \tau$, a discrete analogy of \eqref{HeuristicBalanceLaw}. Exploiting ones again the discrete chain rule as 
$$
\frac{1}{2} \|\bar{\lambda}_\tau(0)\|_{H^{-1}(\Omega; \R^{d+1})}^2 -\frac{1}{2} \| \bar{\lambda}_\tau(t_k)\|_{H^{-1}(\Omega; \R^{d+1})}^2 \dd x\ge
-\int_0^{t_k} \llangle \DT{\lambda}_\tau,\bar{\lambda}_\tau\rrangle\dd t,
$$
we can rewrite \eqref{TestBalanceDisc} as 
\begin{align*}
\mathfrak{G}(t_k, \bar{\nu}_\tau (t_k), \bar{\lambda}_\tau(t_k))  &+ \frac{\varkappa}{2}\|\bar{ \lambda}_\tau (t_k)\|_{H^{-1}(\Omega; \R^{d+1})}^2 \leq  \frac{\varkappa}{2}\|\bar{ \lambda}_\tau (0)\|_{H^{-1}(\Omega; \R^{d+1})}^2 \\&\quad+ \mathfrak{G}(0, \bar{\nu}_\tau (0), \bar{\lambda}_\tau(0)) + \int_0^{t_k} \Big( \int_\Omega \DT{h}_\tau \cdot \bar{m}_\tau \dd x - \varkappa \llangle L \bulet \bar{\nu}_\tau, \DT{\lambda}_\tau\rrangle \Big) \dd t
\end{align*}
Estimate on the right hand side (using Young inequality and (\ref{ass}c))
\begin{align*}
\Big|\int_0^{t_k}\!\!\!\int_\Omega \DT{h}_\tau{\cdot}\bar{m}_\tau \dd x \dd t\Big|
&\leq C \int_0^{t_k}\!\!\!\int_\Omega\big(|\DT{h}_\tau|^{p'}
+ |\cdot|^{p} \bulet \bar{\nu}_\tau \big)\dd x \dd t
\leq C \Big( 1+\int_0^{t_k}\!\!\!\int_\Omega |\cdot|^{p} \bulet \bar{\nu}_\tau \dd x \dd t \Big) \\
\Big|\int_0^{t_k} \llangle L \bulet \bar{\nu}_\tau, \DT{\lambda}_\tau\rrangle \dd t \Big| &\leq \int_0^{t_k} \varkappa\|L \bulet \bar{\nu}_\tau\|_{H^{-1}(\Omega; \R^{d+1})} \|\DT{ \lambda}_\tau\|_{H^{-1}(\Omega; \R^{d+1})} \dd t\\ &\leq \int_0^{t_k} \Big(\frac{\e}{4q c_\mathrm{em}} \|\DT{ \lambda}_\tau\|^q_{H^{-1}(\Omega; \R^{d+1})} 
 + C\|L \bulet \bar{\nu}_\tau\|^{q'}_{H^{-1}(\Omega; \R^{d+1})} \Big) \dd t \\ &\leq \int_0^{t_k} \Big(\frac{\epsilon}{4q}\|\DT{\lambda}_\tau\|_{L^q(\Omega; \R^{d+1})}^q 
+ C \int_\Omega |\cdot|^p \bulet \bar{\nu}_\tau \dd x \Big) \dd t + C
\end{align*}
with $p' = \frac{p}{p-1}$ and $c_\mathrm{em}$ the specific constant for which $\|\DT{ \lambda}_\tau\|^q_{H^{-1}(\Omega; \R^{d+1})} \leq c_\mathrm{em} \|\DT{\lambda}_\tau\|_{L^q(\Omega; \R^{d+1})}^q$. Note that we also estimated, thanks to  $q \geq 2$, $\|L \bulet \bar{\nu}_\tau\|_{H^{-1}(\Omega; \R^{d+1})}^{q'} \leq C (1+\|L \bulet \bar{\nu}_\tau\|_{L^{2}(\Omega; \R^{d+1})}^{2}) \leq C(1+\int_\Omega|L|^{2} \bulet\bar{\nu}_\tau \dd x) \leq C(1+\int_\Omega|\cdot|^{p} \bulet\bar{\nu}_\tau \dd x)$. 

On the other hand, we may estimate $\mathfrak{G}(t_k, \bar{\nu}_\tau (t_k), \bar{\lambda}_\tau(t_k))$ from above by using (\ref{ass}a) as
\begin{align}
\mathfrak{G}(t_k, \bar{\nu}_\tau (t_k), \bar{\lambda}_\tau(t_k)) &\geq \int_\Omega \phi \bulet \bar{\nu}_\tau (t_k) \dd x - \Big|\int_\Omega {h}(t_k){\cdot}\bar{m}_\tau(t_k) \dd x\Big| \nonumber \\ &\geq \int_\Omega c_1^A \big(1+ |\cdot|^p \bulet \bar{\nu}_\tau(t_k)\big) \dd x - \int_\Omega \Big( C|{h}|^{p'}
+ \frac{c_1^A}{2}|\cdot|^{p} \bulet \bar{\nu}_\tau \Big) \dd x  \nonumber \\ &\geq C\Big(\int_\Omega |\cdot|^p \bulet \bar{\nu}_\tau(t_k) \dd x - 1\Big) \label{UpperEstimateG}
\end{align}
Combining \eqref{UpperEstimateG} with the estimates on the right-hand side of \eqref{TestBalanceDisc} one gets 
\begin{align}
C \int_\Omega |\cdot|^p \bulet \bar{\nu}_\tau(t_k)
+ \frac{\varkappa}{2} \int_\Omega |\bar{\lambda}_\tau(t_k)|^{2q} \dd x
\leq \int_0^{t_k}\!\!\int_\Omega \Big(\frac{\e}{4q}|\DT{\lambda}_\tau|^q + C|\cdot|^{p} \bulet \bar{\nu}_\tau \Big)\, \dd x \dd t + C,
\label{finalEstimateEnergy}
\end{align}
a discrete analogy of \eqref{heuristics-1}.
\vspace{1ex}
}

\noindent \emph{Step 2: Testing the flow-rule by $\DT{ \lambda}_\tau$}:
In this step we perform the procedure that on the heuristic level lead to \eqref{heuristics-2}; together with Step 1 and Step 3 it will give \eqref{Nu-Apriori}--\eqref{EnthalpyLinftyL1}.

First note that, as $(\nu_\tau^l,\lambda_\tau^l)$ is a minimizer of \eqref{BalanceEqDis}, the partial sub-differential of the cost functional with respect to $\lambda$ has to be zero at $\lambda_\tau^l$. Realizing that this condition holds at each time level and summing these conditions up to some $k$ leads to
\begin{align}
&\!\!\int_0^{t_k}\!\!\int_\Omega \Big(\delta_S^*( \DT{
\lambda}_\tau)+\frac{\epsilon}{q}|\DT{\lambda}_\tau|^q \Big) \dd x \dd t
\leq \int_0^{t_k}\!\!\bigg(\varkappa\llangle \bar{\lambda}_\tau{-}L \bulet\bar{\nu}_\tau , v_\tau{-}\DT{\lambda}_\tau \rrangle
\nonumber \\&\quad+ \int_\Omega \Big(\big(\mathcalI(\underline{w}_\tau)-\theta_{\rm c}\big) \ccoupl{\cdot}(v_\tau{-}\DT{\lambda}_\tau) + 2q\tau |\bar{\lambda}_\tau|^{2q-2}\bar{\lambda}_\tau(v_\tau{-}\DT{\lambda}_\tau)
+\delta_S^*(v_\tau)+\frac{\epsilon}{q} |v_\tau|^q \Big)\dd x\bigg) \dd t
\label{FlowRuleDisI},
\end{align}
where $v_\tau$ is an arbitrary test function such that $v_\tau(\cdot, x)$ is piecewise constant on the intervals $(t_{j-1}, t_j]$ and $v_\tau(t_j,\cdot) \in L^{2q}(\Omega; \R^{d+1})$ for every $j$.
As
\begin{align*}
&\frac{1}{2} \|\bar{\lambda}_\tau(0)\|_{H^{-1}(\Omega; \R^{d+1})}^2 -\frac{1}{2} \| \bar{\lambda}_\tau(t_k)\|_{H^{-1}(\Omega; \R^{d+1})}^2 \dd x\ge
-\int_0^{t_k} \llangle \DT{\lambda}_\tau,\bar{\lambda}_\tau\rrangle\dd t
\quad\text{ and}\\
&\int_{\Omega}|\bar{\lambda}_\tau(0)|^{2q}-\int_{\Omega}|\bar{\lambda}_\tau(t_k)|^{2q} \dd x \geq -2q\int_0^{t_k}\!\!\int_\Omega \DT{\lambda}_\tau\bar{\lambda}_\tau|\bar{\lambda}_\tau|^{2q-2} \dd x \dd t,
\end{align*}
hold by the discrete chain rule (thanks to the convexity of the involved functions on the left-hand-side), we may test \eqref{FlowRuleDisI} by $v_\tau=0$, which effectively executes the test of the discrete version
of the inclusion \eqref{FlowRuleBasic} by $\frac{1}{q}\DT{ \lambda}_\tau$. This yields
\begin{align}
\int_0^{t_k}\!\!\int_\Omega &\big(\,\delta_S^*( \DT{\lambda}_\tau) +\frac{\epsilon}{q} |\DT{ \lambda}_\tau |^q \big)\dd x \dd t + \frac{\varkappa}{2}\|\bar{ \lambda}_\tau (t_k)\|_{H^{-1}(\Omega; \R^{d+1})}^2 + \int_\Omega \tau |\bar{\lambda}_\tau(t_k)|^{2q} \dd x \nonumber \\ \leq &
\int_0^{t_k}\!\!\bigg(\varkappa\|L\bulet\bar{\nu}_\tau\|_{H^{-1}(\Omega; \R^{d+1})}
\|\DT{ \lambda}_\tau\|_{H^{-1}(\Omega; \R^{d+1})} +\int_\Omega \big(|\mathcalI(\underline{w}_\tau)-\theta_{\rm c}|\,|\ccoupl| \, |\DT{\lambda}_\tau| \big)\dd x\bigg)\dd t\nonumber
\\
& \qquad\qquad\qquad\qquad\qquad
+\frac{\varkappa}{2}\|\bar{\lambda}_\tau (0)\|_{H^{-1}(\Omega ; \R^{d+1})}^2 + \int_\Omega \tau |\bar{\lambda}_\tau (0)|^{2q} \dd x \nonumber \\ \leq & \int_0^{t_k}\!\!\bigg( C\|L\bulet\bar{\nu}_\tau\|^{q'}_{H^{-1}(\Omega; \R^{d+1})} + \frac{\epsilon}{8qc_\mathrm{em}}\|\DT{ \lambda}_\tau\|^q_{H^{-1}(\Omega; \R^{d+1})} \nonumber
\\
& \qquad\qquad\qquad\qquad\qquad
+ \int_\Omega \Big(C|\mathcalI(\underline{w}_\tau)-\theta_{\rm c}|^{q'} +\frac{\epsilon}{8q}|\DT{\lambda}_\tau|^q \Big) \dd x\bigg)\dd t+C \nonumber \\ \leq & \int_0^{t_k}\!\!\int_\Omega \Big( C |\cdot|^{p}\bulet\nu_\tau+C|\bar{w}_\tau| +\frac{\epsilon}{4q}|\DT{ \lambda}_\tau|^q \Big)\dd x\dd t+ \tau\|\lambda_{0,\tau}\|_{L^{2q}(\O; \R^{d+1})}^{2q} + C,
\label{LambdaAprioriOne}
\end{align}
by applying Young's inequality to the terms $\varkappa\|L \bulet \bar{\nu}_\tau\|_{H^{-1}(\Omega; \R^{d+1})} \|\DT{ \lambda}_\tau\|_{H^{-1}(\Omega; \R^{d+1})}$ as well as $|\mathcalI(\underline{w}_\tau)-\theta_{\rm c}||\ccoupl| |\DT{ \lambda}_\tau|$. Subsequently we estimated $\|L \bulet \bar{\nu}_\tau\|_{H^{-1}(\Omega; \R^{d+1})}^{q'}$ the same way ay in Step 1; $c_\mathrm{em}$ was also chosen like in Step 1. Eventually, $|\mathcalI(\underline w_\tau)|^{q'} \leq C(1+|\underline{ w}_\tau|^{q'/\omega}) \leq C(1+|\underline{ w}_\tau|)$
due to \eqref{OmegaGrowthBound} and ${q'/\omega}<1$ (cf. (\ref{ass}d)).
\vspace{1ex}

\noindent \emph{Step 3: Testing the heat equation by 1}:
In this step we perform the procedure that on the heuristic level lead to \eqref{heuristics-2}; together with Step 1 and Step 3 it will give \eqref{Nu-Apriori}--\eqref{EnthalpyLinftyL1}.
Summing the discrete version of the enthalpy equation \eqref{EnthalpyEqDis} from 0 to $t_k$ leads to
\begin{align}\nonumber
&\int_0^{t_k}\!\!\bigg(\int_{\Omega} \Big(\DT{w}_\tau \varphi+\mathcal{K}(\bar{\lambda}_\tau,\bar{w}_\tau)\nabla\bar{w}_\tau{\cdot}\nabla\varphi
\Big)\dd x +\int_{\Gamma}\bar{b}_\tau \mathcalI(\bar{w}_\tau)\varphi\dd S\bigg) \dd t
\\& \qquad\qquad = \int_0^{t_k}\!\!\bigg(\int_\Omega \Big(\delta_S^*(\DT{\lambda}_\tau)
+\epsilon|\DT{\lambda}_\tau|^q +\mathcalI(\bar{w}_\tau)\ccoupl{\cdot}\DT{\lambda}_\tau\Big)\varphi\dd x
+\int_{\Gamma}\bar{b}_\tau\bar{\theta}_{\mathrm{ext}, \tau}\varphi\dd S\bigg)\dd t,
\label{EnthalpyEqDisI}
\end{align}
where $\varphi$ is an arbitrary test function, such that $\varphi(\cdot, x)$ is piecewise constant on the intervals $(t_{j-1}, t_j]$ and $\varphi(t_j,\cdot) \in H^1(\Omega)$ for every $j$.
Then we test equation \eqref{EnthalpyEqDisI} by 1 to get
\begin{align*}
\int_0^{t_k}\!\!\bigg(\int_\Omega \DT{w}_\tau \dd x + \int_{\Gamma} \bar{b}_\tau\mathcalI(\bar{ w}_\tau) \dd S \bigg)\dd t \leq \int_0^{t_k}\!\!\bigg(\int_\Omega \Big( \delta_S^*(\DT{ \lambda}_\tau) + \epsilon|\DT{ \lambda}_\tau|^q + |\mathcalI(\bar{ w}_\tau )|\,|\ccoupl|\, |\DT{ \lambda}_\tau| \Big) \dd x
\\+ \int_{\Gamma} |\bar{b}_\tau \bar{\theta}_{\mathrm{ext},\tau}| \dd S\bigg) \dd t.
\end{align*}
Estimate the third term on the right-hand side similarly as in Step 2 to arrive at the expression
\begin{equation*}
\int_0^{t_k} \int_\Omega \DT{w}_\tau (t,x) \dd x + \int_{\Gamma} \bar {b}_\tau \mathcalI(\bar{ w}_\tau ) \dd S \dd t \leq \int_Q 2|\DT{\lambda}_\tau |^q + C |\bar{w}_\tau | \dd x \dd t.
\end{equation*}
Multiplying this by $\e/(8q)$ and adding to \eqref{LambdaAprioriOne} and \eqref{finalEstimateEnergy} already yields \eqref{Nu-Apriori}--\eqref{EnthalpyLinftyL1} by the usage of the discrete Gronwall inequality.
\vspace{1ex}

\noindent \emph{Step 4: Estimation of $\nabla\bar w_\tau$}:
In this step we prove \eqref{EnthalpyLinftyL1+}.

Let us test \eqref{EnthalpyEqDisI} by $\eta(\bar w_\tau)$ where $\eta(w)=1-\frac{1}{(1+w)^a}$ with $a > 0$, which, due to the non-negativity of the enthalpy is a legal test. Notice that due to the discrete chain rule relying on the convexity of $\tilde{\eta}$
\begin{equation*}
\int_{\Omega} \tilde{\eta}(T)-\tilde{\eta}(0) \dd x =\int_Q\frac{\dd}{\dd t} \tilde{\eta}(\bar{ w}_\tau ) \dd x \dd t \leq \int_Q\DT{ w}_\tau\eta(\bar w_\tau) \dd x \dd t,
\end{equation*}
where $\tilde{\eta}$ denotes the primitive function of $\eta$ such that $\tilde{\eta}(0)=0$. Realize also that due to the fact that $\eta(w) \geq 0$ also $\tilde{\eta}(T) \geq 0$ and hence we may write
\begin{align}
\nonumber &\kappa_0 a \int_Q\frac{|\nabla\bar{ w}_\tau|^2}{(1{+}\bar{w})^{1+a}} \dd x \dd t =\kappa_0\int_Q|\nabla\bar{ w}_\tau|^2\eta'(\bar w_\tau) \dd x \dd t
\le\int_Q\mathcal{K}(\bar{\lambda}_\tau,\bar{w}_\tau)\nabla\bar{w}_\tau {\cdot}\nabla\bar{w}_\tau\eta'(\bar w_\tau)\dd x \dd t
\\ \nonumber&\qquad=
\int_Q\mathcal{K}(\bar{\lambda}_\tau,\bar{w}_\tau)\nabla\bar{w}_\tau
{\cdot}\nabla\eta(\bar w_\tau)\dd x \dd t
\\ \nonumber&\qquad
\le\int_\O\tilde{\eta}(w_\tau(T))\dd x + \int_Q\mathcal{K}(\bar{\lambda}_\tau,\bar{w}_\tau)\nabla\bar{w}_\tau
{\cdot}\nabla\eta(\bar w_\tau)\dd x \dd t +\int_\Sigma \bar{b}_\tau\mathcalI(\bar{w}_\tau )\eta(\bar w_\tau)\dd S \dd t
\\ \nonumber &\qquad
=\int_\Sigma \bar{b}_\tau \bar{\theta}_{\mathrm{ext}, \tau}\eta(\bar w_\tau)\dd S \dd t +\int_\O\tilde{\eta}(w_0)\d x +\int_Q(\delta_S^*(\DT{\lambda}_\tau)+\mathcalI(\bar{w}_\tau)\ccoupl \cdot\DT{\lambda}_\tau +\epsilon|\DT{\lambda}_\tau |^q)\eta(\bar w_\tau) \dd x \dd t \\ \label{nonlin-estimate} &\qquad
\le C+\|\bar r_\tau\|_{L^1(Q)}
\end{align}
where we used the obvious bound $|\eta(\bar w_\tau)|\le1$ and abbreviated
\begin{equation*}
\bar r_\tau:= \delta_S^*(\DT{\lambda}_\tau)+\epsilon|\DT{\lambda}_\tau|^q
+\mathcalI(\bar{w}_\tau )\ccoupl \cdot \DT{ \lambda}_\tau.
\end{equation*}
As to the $L^1$-bound of $\bar{r}_\tau$, realize that
\begin{align*}
&\int_Q\delta_S^*( \DT{ \lambda}_\tau )+ \epsilon|\DT{ \lambda}_\tau |^q \dd x \dd t \leq 2\epsilon \| \DT{ \lambda}_\tau \|_{L^q(Q)}^q + C\quad\text{ and}
\\
&\int_Q\left|\mathcalI(\bar{ w}_\tau )\ccoupl \cdot \DT{ \lambda}_\tau \right| \dd x \dd t \leq C \Big( 1+\|\bar{ w}_\tau (t)\|_{L^1(Q)} +\epsilon \| \DT{ \lambda_\tau} \|_{L^q(Q)}^q \Big),
\end{align*}
similarly as in Step 1 or Step 2 of this proof, and use the already shown estimates \eqref{LambdaApriori}, \eqref{EnthalpyLinftyL1}. Thus, \eqref{nonlin-estimate} yields $\int_Q\frac{|\nabla\bar{ w}_\tau|^2}{(1+\bar{w})^{1+a}} \dd x \dd t$ bounded. Combining it with \eqref{EnthalpyLinftyL1} like in \cite{boccardo0,boccardo-gallouet1}, cf.~also
\cite[Formulae (4.29)-(4.33)]{tr1}, we obtain \eqref{EnthalpyLinftyL1+}.
\vspace{1ex}

\noindent \emph{Step 5: ``Dual'' estimate for the time derivative}:
Notice that
\begin{align*}
\big\| \DT{w }_\tau &\big\|_{\mathcal{M}([0,T], {W^{1,\infty}(\Omega)}^*)}=
\sum_{k=1}^N\Big\|\frac{w_\tau ^k{-}w_\tau^{k-1}}{\tau}\Big\|_{{W^{1,\infty}(\Omega)}^*} =\sum_{k=1}^N \mathop{\mathrm{sup}}_{v \in W^{1,\infty,}(\Omega),\, \|v\| \leq 1} \int_{\Omega}\frac{w_\tau^k {-}w_\tau ^{k-1}}{\tau}v \dd x \\ &
=\sum_{k=1}^N \mathop{\mathrm{sup}}_{v \in W^{1,\infty}(\Omega), \, \|v\| \leq 1} \int_\Omega \Big(- \mathcal{K}(\lambda_\tau^k, w_\tau^k) \nabla w_\tau^k \cdot \nabla v + \delta_S^*(\DT{\lambda}_\tau(t_k)) v + \epsilon |\DT{\lambda}_\tau(t_k)|^q v
\\&\qquad\qquad\qquad
+ \mathcalI(w^k)\ccoupl \cdot \DT{\lambda}_\tau(t_k) v \Big) \dd x+ \int_{\Gamma} (\bar{b}_\tau \bar{\theta}_{\mathrm{ext}, \tau}-\bar{b}_\tau \mathcalI(w^k)) v \dd S
\\ & \leq \mathop{\mathrm{sup}}_{\tilde{v} \in C([0,T],W^{1,\infty}(\Omega)),\, \|v\| \leq 1} \int_Q \Big(- \mathcal{K}(\bar{\lambda}_\tau, \bar{ w}_\tau ) \nabla \bar{ w}_\tau \cdot \nabla \tilde{v} + \delta_S^*(\DT{\lambda}_\tau) \tilde{v}
\\&\qquad\qquad\qquad + \epsilon |\DT{ \lambda}_\tau|^q \tilde{v} \mathcalI(\bar{w}_\tau )\ccoupl \cdot \DT{ \lambda}_\tau \tilde{v} \Big) \dd x + \int_{\Gamma} (\bar{b}_\tau \bar{\theta}_{\mathrm{ext}, \tau}-\bar{b}_\tau \mathcalI(\bar{ w}_\tau)) \tilde{v} \dd S \dd t.
\end{align*}
Now because of all the preceding steps we may use the H\"{o}lder inequality for all terms on the right-hand side to get estimate \eqref{DualEnthalpyEst}.
\end{proof}

\begin{proposition}[Convergence and Existence]\label{MainConvProp}
Provided \eqref{ass} holds, there exist $(\nu, \lambda, w) \in ({\mathscr Y}^{p}(\O;\R^d))^{[0,T]}\times W^{1,q}([0,T]; L^q(\Omega, \R^{d+1})) \times L^r([0,T], W^{1,r}(\Omega))$ and a sequence $\tau \to 0$ such that
\begin{align}
&\bar{\lambda}_\tau \stackrel{*}{\rightharpoonup} \lambda \, \, \, \text{ in $L^\infty([0,T]; L^q(\Omega, \R^{d+1}))$}
\, \, \, \text{and} \, \, \, \bar{\lambda}_\tau(t) \rightharpoonup \lambda(t) \, \, \, \text{$\forall t \in [0,T]$ in $L^q(\Omega; \R^{d+1})$}, \label{lambdaconv} \\
&\DT{\lambda}_\tau \rightharpoonup \DT{\lambda} \, \, \, \text{ in $L^q(Q; \R^{d+1})$}, \label{dotlambdaconv} \\
&\bar{w}_\tau \rightharpoonup w \, \, \, \text{$L^r([0,T]; W^{1,r}(\Omega))$, $ \textstyle r < \frac{d+2}{d+1}$}\, \, \, \text{and} \, \, \, \bar{w}_\tau \to w \, \, \, \text{in $L^{1}(Q)$}, \label{enthalpyconv}\\
& \bar{m}_\tau \rightharpoonup m \, \, \, \text{ in $L^2(Q; \R^{d})$}, \label{idmomentconv} \\
&L \bulet \bar{\nu}_\tau \to L \bulet \nu \, \, \, \text{ in $L^2([0,T]; H^{-1}(\Omega; \R^{d+1}))$}. \label{Lmomentconv}
\end{align}
Moreover, for each $t \in [0,T]$ there exists a subsequence $\tau_{k(t)}$ such that
\begin{equation}
\bar{\nu}_{\tau_{k(t)}}(t) \stackrel{*}{\rightharpoonup} \nu(t) \, \, \, \text{ in $L^\infty_\mathrm{w}(\Omega; \mathcal{M}(\R^d))$}. \label{YoungMeasConv}
\end{equation}
Every $(\nu, \lambda, w)$ obtained in this way is then is a weak solution of
\eqref{ContinuousSystem}.
\end{proposition}

\begin{remark}
\upshape
Note that the Young measure $\nu$ obtained by \eqref{YoungMeasConv} surely \emph{does not need to be measurable as a function of time}. However exploiting the convexity of the magnetic part of the Gibbs free energy and, in particular, its \emph{strict convexity} in the moments $m \equiv \mathrm{id}\bulet \nu$ and $L \bulet \nu$ we shall prove that, in contrast to $\nu$, the moments $m$ and $L \bulet \nu$ \emph{are measurable in time}.
\end{remark}

\begin{proof}[Proof of Proposition~\ref{MainConvProp}]
Again for lucidity, we divide the proof into six steps.
\vspace{1ex}

\noindent \emph{Step 1: Selection of subsequences}:
{
By the a-priori estimates proved in Proposition \ref{aprioriII} we may find a sequence of $\tau$'s and $(\lambda,w) \in (W^{1,q}([0,T]; L^q(\Omega, \R^{d+1})) \times L^r([0,T], W^{1,r}(\Omega))$ such that \eqref{lambdaconv}--\eqref{enthalpyconv} hold. Moreover, there exists $\Xi \in L^2(Q; \R^{d+1})$ and $\tilde{h} \in L^(2[0,T]{\times}\R^d; \R^{d})$ such that
\begin{align}
L \bulet \bar{\nu}_\tau &\rightharpoonup \Xi && \text{ in $L^2(Q; \R^{d+1})$} \label{LBulNuConv}, \\
\nabla u_{\bar{m}_\tau} &\rightharpoonup \tilde{h} && \text{ in $L^2([0,T]{\times}\R^d; \R^{d})$}. \label{TildeH}
\end{align}
Indeed, by \eqref{LambdaApriori} we know that ${\lambda}_{\tau} $ is  (considering also the integrability of the initial 
condition, cf.\ (\ref{ass}g)) bounded in $W^{1,q}([0,T]; L^q(\Omega; \R^{d+1}))$ and hence converges weakly to some $\lambda$ in this space (for the subsequence selected). As $W^{1,q}(I;L^q(\O,\R^{d+1}))\subset C([0,T];L^q(\Omega, \R^{d+1}))$ we also get that
${\lambda}_{\tau}(t)\rightharpoonup\lambda(t)$ for all $t \in [0,T]$ in $L^q(\Omega, \R^{d+1})$. 

Moreover, $\bar{\lambda}_{\tau}$ converges weakly in ${L^q(Q; \R^{d+1})}$ to $\lambda$ as well, because $\|\lambda_\tau -\bar{\lambda}_\tau \|_{L^q(Q; \R^{d+1})} \leq \tau \|\DT{\lambda}\|_{L^q(Q; \R^{d+1})} \to 0$. Further, again by \eqref{LambdaApriori}, $\bar{\lambda}_\tau$ is bounded in $L^\infty([0,T]; L^q(\Omega; \R^{d+1}))$ and also $\mathrm{BV}([0,T]; L^q(\Omega; \R^{d+1}))$. Hence, by making use of a slight modification of Helly's theorem \cite{mielke-diff, mielke-roubicek-stefanelli} $\bar{\lambda}_\tau(t) \rightharpoonup \lambda(t)$ for all $t \in [0,T]$ in $L^q(\Omega;\R^{d+1})$, too. Note also, that due to the fact that $q \geq 2$ and the compact embedding $L^2(\Omega; \R^{d+1}) \Subset H^{-1}(\Omega; \R^{d+1})$, $\bar{\lambda}_\tau(t) \to \lambda(t)$ strongly in $H^{-1}(\Omega; \R^{d+1})$ for all $t \in [0,T]$. 

For this sequence of $\tau$'s, $L \bulet \bar{\nu}_\tau$ is bounded in $L^2(Q; \R^{d+1})$ due to \eqref{Nu-Apriori}; therefore \eqref{LBulNuConv} follows just by Banach's selection principle. Similarly, also $ \bar{m}_\tau$ is bounded in $L^2(Q; \R^{d})$ due to \eqref{Nu-Apriori} and hence, by standard theory for the elliptic equation \eqref{MaxwellDis}, $\nabla u_{\bar{m}_\tau}$ is bounded in $L^2([0,T] \times\R^d; \R^{d})$
; this readily gives \eqref{TildeH}.
}
Due to \eqref{EnthalpyLinftyL1}, $\bar{w}_\tau $ is bounded in $L^{r}([0,T]; W^{1,r}(\Omega))$, $r<\frac{d+2}{d+1}$ and therefore converges weakly to some $w$ in this space. Having the dual estimate on the time derivative of $w_\tau$ \eqref{DualEnthalpyEst}, we exploit the Aubin-Lions-lemma generalized for
measure-valued derivatives (see \cite{necas})
to get that $\bar{w}_\tau $ converges \emph{even strongly} to $w$ in $L^{(d+2)/(d+1)-\beta}([0,T]; W^{1-\beta,(d+2)/(d+1)-\beta}(\O))$ and after interpolation $L^{\frac{d+2}{d}-\beta}(Q)$ for any $\beta>0$ small, so that the traces converge strongly in $L^{(d+2)/(d+1)-\beta}([0,T];L^{(d^2+d-2)/(d^2-2)-\beta}(\Gamma))$; for any small $\beta > 0$ cf.~\cite[Formulae (4.42) and (4.55)]{tr1}.
Moreover the estimate \eqref{DualEnthalpyEst} assures that $\underline{w}_\tau$, cf.~\eqref{w-backward}, converges strongly in $L^{\frac{d+2}{d}-\beta}(Q)$ \emph{to the same limit} as $\bar{w}_\tau$.

Thanks to the growth condition $|\mathcalI(w)| \leq |w|^{1/\omega}$ and assumption \eqref{Sub6Growth} we have that $|\mathcalI(w)| \leq |w|^{1/\omega} \leq |w|^{1/q'}$. Hence exploiting the continuity of Nemytskii mappings in Lebesgue spaces and using continuity of $\mathcalI$, we have that $\mathcalI(\bar{w}_\tau) \to \mathcalI(w)$ in $L^{q'}(Q)$. Similarly also $\mathcalI(\underline{w}_\tau) \to \mathcalI(w)$ in $L^{q'}(Q)$.

Now, take any $t \in [0,T]$ arbitrary but fixed. Then due to the bound \eqref{Nu-Apriori} select a subsequence $\tau_{k(t)}$ such that $\bar\nu_{\tau_{k(t)}}(t) \rightharpoonup \nu(t)$ in $L^\infty_\mathrm{w}(\Omega; \mathcal{M}(\R^d))$. As indicated by the index $k=k(t)$, this selection \emph{may depend on time $t$}. Estimate \eqref{Nu-Apriori} then also assures that $\nu$ is a Young measure. Since the growth of $L$ is strictly smaller than $p$ it holds also that $L \bulet \bar\nu_{\tau_{k(t)}}(t) \rightharpoonup L \bulet \nu(t)$ in $L^2(\Omega; \R^{d+1})$ and therefore also $L \bulet \bar\nu_{\tau_{k(t)}}(t) \to L \bulet \nu(t)$ strongly in $H^{-1}(\Omega; \R^{d+1})$. Similarly $\nabla u_{\mathrm{id} \bulet \bar\nu_{\tau_{k(t)}}(t)} \rightharpoonup \nabla u_{\mathrm{id} \bulet \nu(t)}$ in $L^2(\Omega; \R^{d})$. Note that this does not necessarily imply that $L \bulet \nu = \Xi$ or that $\nabla u_{\mathrm{id} \bulet \nu} =\tilde{h}$ (where $\Xi$ and $\tilde{h}$ were defined above).
\vspace{1ex}

\noindent \emph{Step 2: Minimization principle}:
Let $t$ be still fixed. A direct consequence of \eqref{BalanceEqDis} is the
\emph{discrete minimization principle} that reads as
\begin{equation}
\mathfrak{G} (t_{\tau_{k(t)}}, \bar\nu_{\tau_{k(t)}}(t),
\bar\lambda_{\tau_{k(t)}}(t))\le
\mathfrak{G}(t_{\tau_{k(t)}},\hat{\nu}, \bar\lambda_{\tau_{k(t)}}(t)),
\end{equation}
for any $\hat{\nu} \in {\mathscr Y}^{p}(\O;\R^d)$ with $\mathfrak{G}$ defined in \eqref{def-of-G-mag}. Here we denoted $t_{\tau_{k(t)}} = l \cdot {\tau_{k(t)}}$, where $l = \min_{s \in \mathbb{N}}\{t \leq s{\tau_{k(t)}}\}$. Applying $\liminf_{\tau_{k(t)} \to 0}$ on both sides and using the definition of the magnetic part of the Gibbs free energy we get that
\begin{align*}
&\liminf_{\tau_{k(t)} \to 0} \mathfrak{G} \big(t_{\tau_{k(t)}}, \bar\nu_{\tau_{k(t)}}(t), \bar\lambda_{\tau_{k(t)}}(t)\big)
\leq \lim_{\tau_{k(t)} \to 0} \mathfrak{G}\big(t_{\tau_{k(t)}},\hat{\nu},\bar\lambda_{\tau_{k(t)}}(t)\big)
\\ &\qquad= \lim_{\tau_{k(t)} \to 0} \int_\Omega\phi \bulet \hat{\nu} -h(t_{\tau_{k(t)}}){\cdot}\hat{m}\, \dd x + \big\|\bar\lambda_{\tau_{k(t)}}(t)-L\bulet\hat{\nu}\big\|_{H^{-1}(\Omega; \R^{d+1})}^2 =\, \mathfrak{G}(t,\hat{\nu}, \lambda(t)),
\end{align*}
where $\hat{m} \equiv \id \bulet \hat{\nu}$ because of the continuity of $h$ and the strong convergence of $\bar{\lambda}_\tau(t)$ in $H^{-1}(\Omega; \R^{d+1})$. As to the left-hand side, because of the boundedness from below of $\phi$, we may estimate
\begin{align*}
&\int_\Omega \big(\phi \bulet \nu(t) -h(t){\cdot}m
(t) \big)\,\dd x + K\big\|\lambda(t) - L \bulet \nu(t)\big\|_{H^{-1}(\Omega; \R^{d+1})}^2 \\ &\qquad\le\liminf_{\tau_{k(t)} \to 0} \int_\Omega\phi\bulet\bar{\nu}_{\tau_{k(t)}}(t)-h(t_{\tau_{k(t)}}){\cdot}\bar{m}_{\tau_{k(t)}}(t)
\,\dd x +
K\big\|\bar\lambda_{\tau_{k(t)}}(t)-L\bulet\bar{\nu}_{\tau_{k(t)}}(t)
\big\|_{H^{-1}(\Omega; \R^{d+1})}^2,
\end{align*}
which already gives the sought minimization principle.
\vspace{1ex}

\noindent \emph{Step 3: Measurability of $L \bulet \nu$ and $\nabla u_m(\cdot)$, strong convergence of $L \bulet \nu$}:
Let $t$ remain fixed as in Step~2. Then, the convexity of $\mathfrak{G}$, the convexity of the space of Young measures and the strict convexity of $\mathfrak{G}$ in the terms $\lambda{-}L\bulet\nu$ and $\nabla u_{m}$, used for $m=\mathrm{id} \bulet \nu$ with $\nu$ corresponding to some minimizer of \eqref{balanceLaw} and $\lambda=\lambda(t)$ already selected in Step~1, ensures that both $\lambda{-}L\bulet\nu$ and $\nabla u_{m}$ are determined uniquely. Then, since $\lambda(t)$ is fixed, also $L\bulet\nu$ is determined uniquely, although the minimizer $\nu$ does not need to be.

In turn it means that $L\bulet\bar{\nu}_\tau(t)\rightharpoonup
L\bulet\nu(t)$ in $L^2(\Omega;\R^{d+1})$ and $\nabla u_{\bar{m}_\tau(t)}\rightharpoonup \nabla u_{m)}$ in $L^2(\Omega;\R^{d})$ for the sequence of $\tau$'s already selected in Step 1. Hence, by usage of the Lebesgue dominated convergence theorem, $L \bulet \nu(t) = \Xi(t)$ and $\tilde h = \nabla u_{m (t)}$ for a.a. $t \in [0,T]$. This in particular shows that both $L \bulet \nu$ and $\nabla u_ {m(t)}$ are measurable.

Note that the above implies also $\|L \bulet \bar{\nu}_\tau(t) - L \bulet \nu(t)\|_{H^{-1}(\Omega; \R^{d+1})} \to 0$ for all $t \in [0,T]$. Realizing moreover that $\sup_{t \in [0,T]} \|L \bulet \bar{\nu}_\tau(t) - L \bulet \nu(t) \|^2_{H^{-1}(\Omega; \R^{d+1})} \dd t$ is bounded by a constant due to \eqref{Nu-Apriori} we may use \emph{Lebesgue dominated convergence theorem} to get \eqref{Lmomentconv}.
\vspace{1ex}

\noindent \emph{Step 4: Convergence of the flow rule}:
We are now in the position to pass to the limit in the discrete flow rule,
i.e. \ \eqref{FlowRuleDisI} but integrated to $T$ instead of $t_k$.
First choose a test function $v \in L^q(Q;\R^{d+1})$ and consider its piecewise constant approximations $v_\tau$ such that $v_\tau \to v$ strongly on $L^q(Q;\R^{d+1})$ and moreover $\|v_\tau \|_{L^{2q}(Q; \R^{d+1})} \leq C \tau^{-1/(2q+1)}$.
By convexity of $\delta_S^*(\cdot) +\frac{\epsilon}{q} |\cdot|^{q}$, we get that
\begin{equation*}
\int_Q \big(\delta_S^*(\DT{\lambda}) +\frac{\epsilon}{q} |\DT{\lambda}|^{q} \big) \dd x \dd t \leq \liminf_{\tau \to 0} \int_Q \big(\delta_S^*(\DT{\lambda}_\tau) +\frac{\epsilon}{q} |\DT{\lambda}_\tau|^{q} \big) \dd x \dd t + \int_\Omega \tau |\bar{\lambda}_\tau(T)|^{2q} \dd x.
\end{equation*}

As to the convergence right-hand-side of \eqref{FlowRuleDisI}, we use that $\mathcalI(\underline{w}_\tau) \to \mathcalI(w)$ in $L^{q'}(Q)$ to pass to the limit in $\int_Q \mathcalI(\underline{w}_\tau) {-}\theta_{\rm c})\ccoupl{\cdot}(v_\tau{-}\DT{\lambda}_\tau)\dd x\dd t$
and \eqref{lambdaconv} as well as \eqref{Lmomentconv} to establish the convergence of $\int_0^{T}\varkappa\llangle \bar{\lambda}_\tau{-}L \bulet \bar{\nu}_\tau ,v_\tau {-} \DT{ \lambda}_\tau \rrangle \dd t$. Since $\|\lambda_{0,\tau}\|_{L^{2q}(\Omega; \R^{d+1})}\leq C \tau^{-1/(2q+1)}$ the term $ \tau \int_\Omega |\lambda_0|^{2q} \dd x$ converges to zero. Similarly $2q \tau |\bar{\lambda}_\tau|^{2q-2}\bar{\lambda}_\tau v_\tau$ can be pushed to zero thanks to \eqref{lambdaBlowUp} and the blow-up for $v_\tau$ specified above that allow us to estimate $|\int_Q 2q \tau |\bar{\lambda}_\tau|^{2q-2}\bar{\lambda}_\tau v_\tau|\dd x\dd t
\leq 2q \tau \| \bar{\lambda}_\tau\|^{2q-1}_{L^{2q}(Q; \R^{d+1})} \|v_\tau\|_{L^{2q}(Q; \R^{d+1})} \leq C \tau^{\frac{1}{4q^2+2q}}$.
Altogether, applying $\liminf_{\tau \to 0}$ to both sides of the discrete flow rule, we get
\begin{align}
\int_Q \delta_S^*( \DT{ \lambda})+\frac{\epsilon}{q} |\DT{\lambda}|^q \dd x \dd t
\leq \int_0^{T} \bigg( &\varkappa\llangle \lambda{-}L \bulet \nu,v {-}\DT{ \lambda}\rrangle \nonumber 
\\ &+ \int_\Omega \Big((\mathcalI(w){-}\theta_{\rm c})\ccoupl{\cdot}(v{-}\DT{\lambda})+\delta_S^*(v)+\frac{\epsilon}{q}|v|^q \Big)\dd x \bigg) \dd t,
\label{FlowRuleConverged}
\end{align}
for any $v \in L^q(Q; \R^{d+1})$.
\vspace{1ex}

\noindent \emph{Step 5: Strong convergence of $\DT{\lambda}_\tau$}:
First test the discrete flow rule (cf.\ \eqref{FlowRuleDisI}, \emph{reformulated using the convexity of $|\cdot|^q$}) by  $\DT{\lambda}_{\mbox{\tiny S},\tau}$ being a piecewise constant approximation of the function $\DT{\lambda}$ such that $\DT{\lambda}_{\mbox{\tiny S},\tau} \to \DT{\lambda}$ strongly in $L^q(Q; \R^{M+1})$ and moreover $\|\DT{\lambda}_{\mbox{\tiny S},\tau} \|_{L^{2q}(Q; \R^{M+1})}\leq C \tau^{-1/(2q+1)}$. We get
\begin{align}
&\int_Q \delta_S^*( \DT{ \lambda}_\tau) \dd x\dd t +\int_\Omega \tau |\bar{\lambda}_\tau(T)|^{2q} \dd x \leq \int_\Omega \tau |\lambda_0|^{2q} \dd x
+ \int_0^{T}\!\!\bigg(\varkappa\llangle\bar{\lambda}_\tau-L \bulet\bar{\nu}_\tau,
\DT{\lambda}_{\mbox{\tiny S},\tau}-\DT{ \lambda}_\tau \rrangle \nonumber \\ &
+ \int_\Omega\! \Big(\delta_S^*(\DT{\lambda}_{\mbox{\tiny S},\tau}) + \epsilon |\DT{\lambda}_\tau|^{q-2}\DT{\lambda}_\tau{\cdot}(\DT{\lambda}_{\mbox{\tiny S},\tau} {-} \DT{\lambda}_\tau)+ \big(\mathcalI(\underline{w}_\tau){-}\theta_{\rm c}\big) \ccoupl {\cdot}(\DT{\lambda}_{\mbox{\tiny S},\tau} {-} \DT{\lambda}_\tau)+ 2q\tau |\bar{\lambda}_\tau|^{2q-2} \bar{\lambda}_\tau{\cdot}\DT{\lambda}_{\mbox{\tiny S},\tau}\Big)\dd x\bigg) \dd t .
\label{TestStrongConv1}
\end{align}
Symmetrically we test the continuous flow rule reformulated as above
by $\DT{\lambda}_\tau$ to get
\begin{align}\nonumber
\int_Q \delta_S^*( \DT{ \lambda})\dd x \dd t &\leq \int_0^{T}\!\!\bigg(\varkappa\llangle\lambda-L\bulet\nu,\DT{\lambda}_\tau-\DT{ \lambda}\rrangle
\\&\quad +\int_\Omega \Big(\epsilon |\DT{\lambda}|^{q-2}\DT{\lambda}{\cdot}(\DT{\lambda}_\tau-\DT{\lambda})+(\mathcalI(\bar{w})-\theta_{\rm c})\ccoupl{\cdot}(\DT{\lambda}_\tau-\DT{\lambda})+\delta_S^*(v) \Big) \dd x\bigg)\dd t.
\label{TestStrongConv2}
\end{align}
We add \eqref{TestStrongConv1} and \eqref{TestStrongConv2}, apply H\"older inequality and $\lim_{\tau \to 0}$ to estimate
\begin{align*}
\epsilon\lim_{\tau \to 0}&\Big(\|\DT{\lambda}_\tau\|_{L^q(Q; \R^{d+1})}^{q-1}-\|\DT{\lambda}\|_{L^q(Q; \R^{d+1})}^{q-1}\Big)\Big(\|\DT{\lambda}_\tau\|_{L^q(Q; \R^{d+1})} - \|\DT{\lambda}\|_{L^q(Q; \R^{d+1})}\Big)\\
&\le \lim_{\tau \to 0} \epsilon\int_Q\Big(|\DT{\lambda}_\tau|^{q-2}\DT{\lambda}_\tau-|\DT{\lambda}|^{q-2}\DT{\lambda}\Big)\cdot(\DT{\lambda}_\tau-\DT{\lambda}) \dd x \dd t \\ &\leq \lim_{\tau \to 0} \Bigg(\int_Q \Big(\underbrace{\epsilon |\DT{\lambda}_\tau|^{q-2}\DT{\lambda}_\tau(\DT{\lambda}_{\mbox{\tiny S},\tau}-\DT{\lambda})+\delta_S^*(\DT{\lambda}_{\mbox{\tiny S},\tau})-\delta_S^*(\DT{\lambda})}_\mathrm{(I)} \Big) \d x \d t \\ &+ \int_Q \Big( \underbrace{(\mathcalI(\underline{w}_\tau)-\theta_{\rm c})\ccoupl \cdot(\DT{\lambda}_{\mbox{\tiny S},\tau}\!- \DT{\lambda}_\tau)+\mathcalI(w)\ccoupl \cdot(\DT{\lambda}_\tau{-}\DT{\lambda})}_\mathrm{(II)} \Big) \dd x \dd t + \int_0^T\!\!\!\underbrace{\varkappa\llangle \lambda{-}L \bulet \nu,\DT{\lambda}_\tau{-}\DT{\lambda}\rrangle}_\mathrm{(III)} \d t \\&+ \int_0^T \Big( \underbrace{\varkappa\llangle \lambda_\tau-L \bulet \bar{\nu}_\tau,\DT{\lambda}_{\mbox{\tiny S},\tau} - \DT{\lambda}_\tau\rrangle}_\mathrm{(IV)} + \int_\Omega \underbrace{2q \tau |\bar{\lambda}_\tau|^{2q-2} \bar{\lambda}_\tau
\DT{\lambda}_{\mbox{\tiny S},\tau}}_{\mathrm{(V)}} \dd x \Big) \dd t + \int_\Omega \underbrace{\tau |\lambda_0|^{2q}}_{\mathrm{(VI)}} \Bigg) \leq 0.
\end{align*}
When passing to the limit on the right-hand-side use that $\DT{\lambda}_{\mbox{\tiny S},\tau} \to \DT{\lambda}$ in $L^q(Q; \R^{d+1})$ to limit Term (I) to 0.
For Term (II) we use the convergences established in Step~1, and the convergence
of $\mathcalI(\underline{w}_\tau) \to \mathcalI(w)$ in $L^{q'}(Q; \R^{d+1})$ to see
its limit being 0. When turning to Term (III), \eqref{dotlambdaconv} needs to
be applied to find this term approaching again 0. Term (IV) converges to 0 by
applying \eqref{lambdaconv} and \eqref{Lmomentconv} combined with
$\DT{\lambda}_{\mbox{\tiny S},\tau} \to \DT{\lambda}$ in $L^q(Q; \R^{d+1})$.
Term (V) can be pushed to 0 similarly as when converging the flow rule using
the available blow-up conditions, Term (VI) converges also to 0 by exploiting
\eqref{lambdaZeroApprox}. Passing then to the limit and using all above said,
we arrive at
$\|\DT{\lambda}_\tau\|_{L^q(Q; \R^{d+1})}\to\|\DT\lambda\|_{L^q(Q; \R^{d+1})}$. Hence,
by the local convexity of $L^q(Q; \R^{d+1})$, the
already proved weak convergence $\DT{\lambda}_\tau \to \DT{\lambda}$ in
$L^q(Q;\R^{d+1})$ turns to be strong.
\vspace{1ex}

\noindent \emph{Step 6: Convergence of the enthalpy equation}:
Performing discrete by parts integration in \eqref{EnthalpyEqDis}) yields
\begin{align}
-\int_Q \big( \bar{w}_\tau \DT{\varphi} + \mathcal{K}(\bar{\lambda}_\tau, \bar{w}_\tau) \nabla \bar{w}_\tau \nabla \bar{\varphi} \big)\dd x + \int_{\Gamma} \bar{b}_\tau \mathcalI(\bar{w}_\tau)\bar{\varphi} \dd S \dd t \nonumber = \int_\Omega
w_{0,\tau}\bar{\varphi}(0) \dd x
\\ + \int_Q\! \big(\delta_S^*(\DT{\lambda}_\tau) \bar{\varphi} + \epsilon |\DT{\lambda}_\tau|^q \bar{\varphi} \big) \dd x \dd t+\int_Q\!\mathcalI(\bar{w}_\tau)\ccoupl{\cdot}\DT{\lambda}_\tau \bar{\varphi} \dd x+ \int_{\Gamma} \bar{b}_\tau \bar{\theta}_{\mathrm{ext}, \tau}\bar{\varphi} \dd S \dd t,
\end{align}
for all $\bar{\varphi}$ piecewise constant on the intervals $(t_{j-1}, t_j]$ such that $\bar{\varphi}(t_j,\cdot) \in H^1(\Omega)$, $\bar{\varphi}(T) = 0$ and $\varphi$ piecewise linear on the intervals $(t_{j-1}, t_j]$, such that $\bar{\varphi}(t_j,\cdot) = \varphi(t_j,\cdot)$. Note that by such test functions we may approximate (strongly in the norm of $L^p(Q), \, p\in[1,\infty]$) any $\tilde{\varphi} \in C^1(\bar{Q})$.

To make a limit passage for $\tau \to 0$ in this equation, we make use of \eqref{enthalpyconv} (and the approximation of $\varphi$ mentioned above) to handle the term $ \bar{w}_\tau \DT{\varphi}$. Then use that $\bar{\lambda}_\tau \to \lambda$ strongly in $L^q(Q; \R^{d+1})$, which together with \eqref{Sub6Growth} gives $\mathcal{K}(\bar{\lambda}_\tau, \bar{w}_\tau) \to \mathcal{K}(\lambda,w)$ strongly in any Lebesgue space, except for $L^\infty(Q; \R^{d \times d})$.
Exploit also that $\bar{w}_\tau\to w$ strongly in $L^{(d+2)/(d+1)-\beta}([0,T]; W^{1-\beta,(d+2)/(d+1)-\beta}(\O))$
for any $\beta>0$ small (as shown in Step 1), so that the traces converge strongly in $L^{(d+2)/(d+1)-\beta}([0,T];L^{(d^2+d-2)/(d^2-2)-\beta}(\Gamma))$.
Combining that with \eqref{OmegaGrowthBound} allows us to handle the left-hand-side boundary term.
For the right-hand-side we exploit the strong convergence $\DT{\lambda}_\tau \to \DT{\lambda}$ in $L^q(Q; \R^{d+1})$ to the limit in the terms expressing dissipated heat. For the term $\mathcalI(\bar{w}_\tau)\ccoupl {\cdot} \DT{\lambda}_\tau \bar{\varphi}$ exploit $\mathcalI(\bar{w}_\tau) \to \mathcalI(w)$ in $L^{q'}(Q)$, for the right-hand-side boundary term we have weak convergence of $\bar{b}_\tau \bar{\theta}_{\mathrm{ext}, \tau} \rightharpoonup b\theta_\mathrm{ext}$ in $L^1(\Sigma)$, which is enough to establish the limit of this term.
\end{proof}

\section{Less dissipative modification of the model}
In some cases it might be advantageous to generalize the presented model in
the following way: assume that only one part of the internal parameter $\lambda$, which we denote $A \lambda$, to be subjected to rate-dependent dissipation as before while the other part $\lambda - A \lambda$ evolves purely rate-independently; here $A: \R^{d+1} \to \R^{d+1}$ is a linear \emph{projection}. Assume moreover that the coupling between the magnetic and thermic part is realized \emph{only through the rate-dependent part of the vector of volume fractions}, i.e. ${\rm Ker}\,A\subset{\rm Ker}(\ccoupl\cdot)$. We alter naturally the dissipation potential
\begin{equation}
\label{form-of-R1} R(\DT{\lambda}):=\int_\O\delta_S^*\big(\DT{\lambda}\big)+
\frac\epsilon{q}\big|A\DT{\lambda}\big|^q\,\d x.
\end{equation}
We further suppose that the rate-independent part can be split as
$\delta_S^*(\DT{\lambda}) =\delta_{S_2}^*(A\DT{\lambda})
+\delta_{S_1}^*(A^\perp\DT{\lambda})$ with
$A^\perp\DT{\lambda}:=\DT{\lambda}-A\DT{\lambda}$,
and leave the Gibbs free energy unchanged as in \eqref{def-of-psi}. The evolution of the system is then again governed by \eqref{ContinuousSystem}.
Now, however, the system hosts a rate-independent process in evolution
of the component $A^\perp\lambda$.
In the spirit of \cite{tr1} we can formulate this system weakly by exploiting
the concept of \emph{semi-stability} combined with
energy balance:

\begin{definition}
\label{defOfSol-rem}
We call the triple $(\nu, \lambda, w) \in ({\mathscr Y}^p(\Omega; \R^d))^{[0,T]}\times \mathrm{BV}([0,T]; L^1(\Omega; \R^{d+1})) \times
L^1([0,T]; W^{1,1}(\Omega))$ such that $m \equiv \mathrm{id} \bulet \nu \in L^2(Q; \R^d)$, $L\bulet\nu\in L^{2}(Q; \R^{d+1})$ and
$A\DT{\lambda} \in L^q(\Omega; \R^{d+1})$ a \emph{weak solution}
to \eqref{ContinuousSystem1} with $\epsilon|\DT{\lambda}|^{q-2}\DT{\lambda}$
in \eqref{FlowRuleBasic1} and $\epsilon|\DT{\lambda}|^q$ in
\eqref{EnthalpyEquation} replaced respectively
by $\epsilon|A\DT{\lambda}|^{q-2}A\DT{\lambda}$
and $\epsilon|A\DT{\lambda}|^q$
if it
satisfies:
\begin{enumerate}
\item The {\bf semistability}: For all $\tilde \nu$ in
${\mathscr Y}^{p}(\O;\R^d)$, all $\tilde \lambda \in H^{-1}(\Omega;\R^{d+1})$
such that $A\tilde{\lambda} = 0$ and all $t\in[0,T]$
\begin{equation}
\Gm(t,\nu,A^\perp\lambda,\mathcalI(w))
\le\Gm(t,\tilde\nu,\tilde\lambda,\mathcalI(w))
+\int_\Omega\delta_{S_1}^*(A^\perp(\tilde{\lambda}{-}\lambda))\,\dd x.
\label{Rem-balanceLaw}
\end{equation}
\item The {\bf rate dependent flow rule}:\eqref{FlowRule} with $A\DT{\lambda}$ instead of $\DT{\lambda}$ and $Av$ instead of $v$.
\item The {\bf total energy balance}
\begin{align}
\mathfrak{G}(T,\nu(T),\lambda(T))+\int_{\Sigma}b\mathcalI(w)\,\dd S \dd t
\leq \mathfrak{G}(0, \nu_0,\lambda_0)+ \int_\Omega
w_0\dd x \nonumber \\+\int_{\Sigma}\theta_{\mathrm{ext}}\,\dd S\dd t + \int_0^T\!\!\!\mathfrak{G}'_t(t,\nu_u,\bar{\lambda}_\tau)\,\dd t.
\end{align}
\item The {\bf enthalpy equation}: \eqref{EnthalpyEq} with
$\int_{\bar Q}\varphi\,\mathcal{H}(\dd x\dd t) +
\int_Q \delta_{S_2}^*(A\DT{\lambda}) \varphi\dd x\dd t$
instead of $\int_Q \delta_S^*(\DT{\lambda}) \varphi \dd x \dd t$
and $\int_Q \epsilon |A \DT{\lambda}|^q \varphi \dd x \dd t$
instead of $\int_Q \epsilon |\DT{\lambda}|^q \varphi \dd x \dd t$;
here we denoted $\mathcal{H} \in \mathcal{M}(\bar{Q})$ the measure
(=heat production rate by rate-independent dissipation)
defined by prescribing its values on every closed set $A = [t_1,t_2] \times B$, where $B \subset \O$ is a Borel set as
\begin{equation*}
\mathcal{H}(A) = \mathrm{Var}_{\delta_{S_1}^*}(\lambda|_{B}; t_1, t_2).
\end{equation*}
\item {\bf remaining initial conditions} and the {\bf reduced Maxwell system} \eqref{maxwellSysWeak}.
\end{enumerate}
\end{definition}

\begin{theorem}\label{thm-rem}
Let \eqref{ass} hold and $A$ be a projector $\R^{d+1}\to \R^{d+1}$.
Then at least one weak solution to \eqref{ContinuousSystem} in the enthalpy formulation in accord with Definition~\ref{defOfSol-rem} does exist and also \eqref{estimates-of-w} holds but weakened as
$\frac{\partial w}{\partial t}\in L^1(I;W^{1,\infty}(\Omega)^*)+
\mathcal{M}(\bar Q)$.
\end{theorem}

\begin{proof}[Sketch of the proof]
We proceed in a similar way as above by discretizing the problem in time and introducing a time-incremental problem \eqref{BalanceEqDis} with the discrete version of $|\DT{\lambda}|^q$ replaced by $|A \DT{\lambda}|^q$.

The existence of approximate solutions and also a-priori estimates are proved
similarly as above with obvious modifications where there are needed; in
particular Step 1 of Proposition
\ref{aprioriII} gives the boundedness of $\bar{\lambda}_\tau$ in $L^\infty([0,T];H^{-1}(\Omega;\R^{d+1}))\cap
\mathrm{BV}([0,T];L^1(\Omega;\R^{d+1}))$.

As to convergence, even in this case, we can establish obvious modifications of convergences \eqref{lambdaconv}-\eqref{YoungMeasConv} (we replace $\DT{\lambda}$ by $A \DT{\lambda}$ when necessary). To obtain the semistability, we first note that a direct consequence of the time-incremental problem is the \emph{discrete semi-stability condition}
\begin{equation}
\mathfrak{G} (t_{\tau_{k(t)}}, \bar\nu_{\tau_{k(t)}}(t), A^\perp\lambda_{\tau_{k(t)}}(t)) \leq \mathfrak{G}(t_{\tau_{k(t)}},\hat{\nu}, \hat{\lambda}) + \int_\Omega \delta_S^*(A^\perp(\hat{\lambda} - \lambda_{\tau_{k(t)}})) \dd x,
\end{equation}
for any $\hat{\nu} \in {\mathscr Y}^{p}(\O;\R^d)$ and any $\hat{\lambda} \in H^{-1}(\Omega; \R^{d+1})$ such that $A \hat{\lambda} = 0$. This semistability can be converged similarly as above, however as we do not know that $\lambda_{\tau_{k(t)}}(t) \to \lambda(t)$ strongly in $L^1(\Omega; \R^{d+1})$ (would be necessary to establish the convergence of the right-hand-side term $ \delta_S^*(A^\perp(\hat{\lambda} - \lambda_{\tau_{k(t)}}))$ we use that $\mathfrak{G}$ is \emph{quadratic} in $\lambda$ and employ the so-called \emph{binomic trick} as in e.g. \cite{tr1}. This trick is based on choosing the test function as $\hat{\lambda} = \tilde{\lambda} - \lambda(t) + \lambda_{\tau_{k(t)}}(t)$ (where $\tilde{\lambda}$ is arbitrary but such that $A \tilde{\lambda} = 0$), subtracting the right from the left hand side of the semistability, use the binomic formula and then passing to the limit.

The flow rule can be converged in the same manner as above, also we obtain the strong convergence of $A\DT{\lambda}_{\tau} \to A\DT{\lambda}$ in $L^q(Q; \R^{d+1})$. The only delicate part in this case is the convergence of the heat equation, in particular the convergence of the right-hand-side rate-independent dissipation terms. To be able to find the limit of this term, it is necessary to prove the discrete energy inequality
\begin{align}
\mathfrak{G}(T, \bar{\nu}_\tau (T), \bar{\lambda}_\tau (T)) +\int_\Omega \tau|\bar{\lambda}_\tau(T)|^{2q} \dd x+ \int_Q \big(\delta_S^*(\DT{\lambda}_{\tau}) \nonumber + \epsilon |A\DT{\lambda}_{\tau}|^q \big) \dd x \dd t
\\
\leq \mathfrak{G}(0, \bar{\nu}_\tau(0),\bar{\lambda}_\tau(0))+\int_\Omega \tau |\lambda_{\tau, 0}|^{2q} \dd x
+ \int_0^T \mathfrak{G}'_t(t,\underline{ \nu}_\tau,\underline{\lambda}_\tau).
\label{Rem:DisEnergBal}
\end{align}
To get this inequality, we exploit the \emph{convexity of the problem} and realize that solutions of the original time-incremental problem are also solutions of the following auxiliary minimization problem
\begin{align}
\left.\begin{array}{ll}
\mathrm{Minimize} &\Gm(t_k, \nu , \lambda , \mathcalI(w_\tau ^{k-1})) + \displaystyle \int_\Omega \bigg( \tau |\lambda|^{2q} + \tau\delta_S^* \Big(\frac{\lambda -\lambda_\tau ^{k-1}}{\tau}\Big)
\\ &+ \tau \epsilon \Big|\frac{A\lambda^k_\tau -A\lambda^{k-1}_\tau}{\tau}\Big|^{q-2}\Big(\frac{A\lambda^k_\tau-A\lambda_\tau^{k-1}}{\tau}\Big) \Big(\frac{A\lambda-\lambda\tau^{k-1}}{\tau}\Big) \bigg) \dd x\\
\mathrm{subject \, to} &(\nu, \lambda) \in {\mathscr Y}^{p}(\O;\R^d) \times L^{2q}(\Omega; \R^{d+1})\end{array}\right\}
\label{Rem:TIPAlter}
\end{align}
and vice versa. Inequality \eqref{Rem:DisEnergBal} is then got by testing \eqref{Rem:TIPAlter} by $(\nu_\tau^{k-1}, \lambda_\tau^{k-1})$. Using \eqref{Rem:DisEnergBal} and also the ``inverse limit energy inequality"
\begin{align}\nonumber
\mathfrak{G}(T, \nu(T), \lambda(T))+\mathrm{Var}_{\delta_{S_1}^*}(A^\perp\lambda; [0,T])+ \int_Q \delta_{S_2}^*(A\DT{\lambda}) + \epsilon|A \DT{\lambda}|^q \dd x \dd t \\
\geq\mathfrak{G}(0,\nu(0), \lambda(0))
+\int_0^T \mathfrak{G}'_t(t, \nu(t), \lambda(t)) \dd t
\label{energyIneqOther}
\end{align}
that is a consequence of the semistability (see e.g. \cite{tr1}) we can
estimate
\begin{align*}
\int_{\bar Q}&\mathcal{H}(\dd x \dd t)=\mathrm{Var}_{\delta_{S_1}^*}(A^\perp\lambda, [0,T]) \leq \liminf_{\tau \to 0} \int_Q \delta_{S_1}^*(\DT{\lambda}_{\tau}) \dd x \dd t
\leq \limsup_{\tau \to 0} \int_Q \delta_{S_1}^*(\DT{\lambda}_{\tau}) \dd x \dd t \\ &\leq \limsup_{\tau \to \infty} \int_0^T\bigg(
\mathfrak{G}'_t(t,\bar{\nu}_\tau(t))
- \int_\Omega \big( \delta_{S_2}^*(A\DT{\lambda}_\tau) + \epsilon|A \DT{\lambda}_\tau|^q \big)\dd x \bigg)\dd t - \mathfrak{G}(T, \bar{\nu}_\tau(T), \bar{\lambda}_\tau(T)) \\ & \qquad +\mathfrak{G}(0,\nu_0, \lambda_0) - \int_\Omega \big(\tau |\bar{\lambda}_\tau(T)|^{2q} + \tau |\bar{\lambda}_\tau(0)|^{2q}\big) \dd x \\ &
\leq \int_0^T\bigg(\mathfrak{G}'_t(t, \nu(t))
- \int_\Omega \big(\delta_{S_2}^*(A\DT{\lambda})
+ \epsilon|A \DT{\lambda}|^q \big) \dd x\bigg)\dd t
- \mathfrak{G}(T, \nu(T), \lambda(T))+\mathfrak{G}(0,\nu_0, \lambda_0)
\\ &\leq\mathrm{Var}_{\delta_{S_1}^*}(A^\perp\lambda, [0,T]),
\end{align*}
which yields $\lim_{\tau \to 0} \int_Q \delta_{S_1}^*(\DT{\lambda}_\tau) \dd x \dd t = \mathrm{Var}_{\delta_{S_1}^*}(A^\perp\lambda, [0,T])$ and hence allows us to perform the limit passage in the heat equation similarly as above. Note that, due to the form of $\mathfrak{G}$, $\mathfrak{G}'_t$ does not depend on $\lambda$.
\end{proof}

\subsection*{Acknowledgments}
\begin{minipage}{1.0\textwidth}
\baselineskip=10pt
{\small
The authors would like to thank two anonymous referees for helpful suggestions and comments.\\
The authors acknowledge partial support from the grants A~100750802 (GA~AV~\v CR), 106/08/1379, 201/09/0917, 106/09/1573, and 201/10/0357 (GA \v CR), MSM~21620839, VZ6840770021, and 1M06031 (M\v SMT \v CR), 41110 (GAUK \v{C}R) and
MFF-261\,310 and from the research plan AV0Z20760514 (\v CR). This research
was done as an activity within the Ne\v{c}as Center for mathematical modeling
(project LC06052, financed by M\v{S}MT \v{C}R).
}
\end{minipage}

\end{sloppypar}


\begin{thebibliography}{19}
\baselineskip=10pt
\vspace{-.1em}\bibitem{banas-prohl-slodicka}
{\rm L. Ba\v{n}as, A. Prohl, M. Slodi\v{c}ka}: Modeling of thermally assisted magnetodynamics. {\it SIAM J. Num. Anal.} {\bf 47} (2008), 551--574.

\vspace{-.1em}\bibitem{bergqvist}
{\rm A. Bergqvist}: Magnetic vector hysteresis model with dry friction-like
pinning. {\it Physica B} {\bf 233} (1997), 342--347.

\vspace{-.1em}\bibitem{boccardo0}
{\rm L. Boccardo, A. Dall'aglio, T. Gallou\"et, L. Orsina}: Nonlinear
parabolic equations with measure data. {\it J. Funct. Anal.} {\bf 147}
(1997), 237--258.

\vspace{-.1em}\bibitem{boccardo-gallouet1} L. Boccardo, T. Gallou\"et:
Non-linear elliptic and parabolic equations involving measure data.
{\it J. Funct. Anal.} {\bf 87} (1989), 149--169.

\vspace{-.1em}\bibitem{brown3}
{\rm W.F. Brown, Jr.}: {\it Magnetostatic Principles in Ferromagnetism},
Springer, New York, 1966.

\vspace{-.1em}\bibitem{carstensen-prohl}
{\rm C. Carstensen, A. Prohl}: Numerical analysis of relaxed micromagnetics by penalised
finite elements. {\it Numer. Math.} {\bf 90} (2001), 65-99.

\vspace{-.1em}\bibitem{choksi-kohn}
{\rm R. Choksi, R.V. Kohn}: Bounds on the micromagnetic energy of a uniaxial
ferromagnet. {\it Comm. Pure Appl. Math.} {\bf 55} (1998), 259--289.

\vspace{-.1em}\bibitem{dalmaso} G. Dal Maso,G.A. Francfort, R. Toader:
Quasistatic crack growth in nonlinear elasticity,
{\it Arch. Ration. Mech. Anal.} {\bf 176} (2005), pp. 165-225.

\vspace{-.1em}\bibitem{desim}
{\rm A. DeSimone}: Energy minimizers for large ferromagnetic bodies.
{\it Arch. Rat. Mech. Anal.} {\bf 125} (1993), 99--143.

\vspace{-.1em}\bibitem{gilbarg}
{\rm D. Gilbarg, N.S. Trudinger}: {\it Elliptic Partial Differential Equations
of the Second Order.} 2nd ed. Springer, New York, 1983.

\vspace{-.1em}\bibitem{hahn} H. Hahn: {\rm \"{U}ber An\"{a}nherung an Lebesguesche Integrale durch Riemannsche Summen},
Sitzungber. Math. Phys. Kl. K. Akad. Wiss. Wien {\bf 123} (1914), 713-743.

\vspace{-.1em}\bibitem{halphen-nguyen}
B. Halphen, Q. S. Nguyen: Sur les materiaux standards g\'{e}n\'{e}ralis\'{e}s. {\it J.
M\'{e}canique} {\bf 14} (1975), 39-63.

\vspace{-.1em}\bibitem{hubert-schafer}
{\rm A. Hubert, R. Sch\"{a}fer}: {\it Magnetic Domains: The Analysis
of Magnetic Microstructures.} Springer, Berlin, 1998.

\vspace{-.1em}\bibitem{jam-kin}
{\rm R.D. James, D. Kinderlehrer}: Frustration in ferromagnetic
materials. {\it Continuum Mech. Thermodyn.} {\bf 2} (1990), 215--239.

\vspace{-.1em}\bibitem{jam-mu}
{\rm R.D. James, S. M\"uller}: Internal variables and fine scale
oscillations in micromagnetics. {\it Continuum Mech. Thermodyn.} {\bf 6}
(1994), 291--336.

\vspace{-.1em}\bibitem{kruzik-prohl-1}
{\rm M. Kru\v{z}\'{\i}k, A. Prohl}:
Young measure approximation in micromagnetics. {\it Numer. Math.}
{\bf 90} (2001), 291--307.

\vspace{-.1em}\bibitem{kruzik-prohl}
{\rm M. Kru\v{z}\'{\i}k, A. Prohl}: Recent developments in the modeling,
analysis, and numerics of ferromagnetism. {\it SIAM Rev.}, {\bf 48} (2006),
439--483.

\vspace{-.1em}\bibitem{kruzik-roubicek}
{\rm M. Kru\v{z}\'{\i}k, T. Roub\'\i\v cek}:
Specimen shape influence on hysteretic response of bulk ferromagnets.
{\it J. Magnetism and Magn. Mater.} {\bf 256} (2003), 158-167.

\vspace{-.1em}\bibitem{kruzik-roubicek-2}
{\rm M. Kru\v{z}\'{\i}k, T. Roub\'\i\v cek}: Interactions between demagnetizing
field and minor-loop development in bulk ferromagnets.
{\it J. Magnetism and Magn. Mater.} {\bf 277} (2004), 192-200.

\vspace{-.1em}\bibitem{landau-lifshitz}
{\rm L.D. Landau, E.M. Lifshitz}: On theory of the dispersion
of magnetic permeability of ferromagnetic bodies. {\it Physik Z.
Sowjetunion} {\bf 8} (1935), 153--169.

\vspace{-.1em}\bibitem{landau-lifshitz-1}
{\rm L.D. Landau, E.M. Lifshitz}: {\it Course of Theoretical Physics.}
vol. {\bf 8}, Pergamon Press, Oxford, 1960.

\vspace{-.1em}\bibitem{luskin-ma}
{\rm M. Luskin, L. Ma}: Numerical optimization of the micromagnetics energy. In: {\it Proceedings of the Session on ``Mathematics in Smart Materials''}, SPIE, {\bf 1919} (1993), pp. 19--29.

\vspace{-.1em}\bibitem{mielke-diff}
{\rm A. Mielke}: Evolution of rate-independent systems. In:
{\it Handbook of Differential Equations: Evolutionary Diff. Eqs.}
(Eds. C.Dafermos, E.Feireisl), Elsevier, Amsterdam, 2005.
pp.461--559.

\vspace{-.1em}\bibitem{mielke-roubicek}
A. Mielke, T. Roub\'\i\v cek : Rate-independent model of
inelastic behaviour of shape-memory alloys.
{\it Multiscale Modeling Simul.} {\bf 1} (2003), 571-597.

\vspace{-.1em}\bibitem{mielke-roubicek-stefanelli}
A. Mielke, T. Roub\'\i\v cek, U. Stefanelli: $\Gamma$-limits and relaxations
for rate-independent evolutionary problems. {\it Calc. Var.} {\bf 31}
(2008), 387--416.

\vspace{-.1em}\bibitem{mielke-theil-levitas}
{\rm A. Mielke, F. Theil, V.I. Levitas}: Mathematical formulation of
quasistatic phase transformations with friction using an extremum principle.
{\it Arch. Rat. Mech. Anal.} {\bf 162} (2002), 137--177.

\vspace{-.1em}\bibitem{necas}
{\rm J. Ne\v cas, T.Roub\'\i\v cek}:
Buoyancy-driven viscous flow with $L^1$-data. {\it Nonlinear Anal.} {\bf 46}
(2001), 737-755.

\vspace{-.1em}\bibitem{pedregal0}
{\rm P. Pedregal}: Relaxation in ferromagnetism: the rigid case,
{\it J. Nonlinear Sci.} {\bf 4} (1994), 105--125.

\vspace{-.1em}\bibitem{pedregal}
{\rm P. Pedregal}: {\it Parametrized Measures and Variational Principles},
Birkh\"{a}user, Basel, 1997.

\vspace{-.1em}\bibitem{PedregalYan}
{\rm P. Pedregal, B. Yan}: {A Duality Method for Micromagnetics}, {\it SIAM J. Math. Anal.} {\bf 41} (2010), 2431--2452.

\vspace{-.1em}\bibitem{podio-roubicek-tomassetti}
{\rm P. Podio-Guidugli, T. Roub\'\i\v cek, G. Tomassetti}:
A thermodynamically-consistent theory of the ferro/paramagnetic transition.
{\it Arch. Rat. Mech. Anal.} {\bf 198} (2010), 1057-1094.

\vspace{-.1em}\bibitem{rogers}
{\rm R.C. Rogers}: A nonlocal model for the exchange energy in ferromagnetic
materials. {\it J. Int. Eq. Appl.} {\bf 3} (1991), 85--127.

\vspace{-.1em}\bibitem{roubicek0}
{\rm T. Roub\'\i\v cek}: {\it Relaxation in Optimization Theory and
Variational Calculus}. W.~de Gruyter, Berlin, 1997.

\vspace{-.1em}\bibitem{tr1}
T. Roub\'\i\v cek: Thermodynamics of rate independent processes in
viscous solids at small strains. {\it SIAM J. Math. Anal.}
{\bf 40} (2010), 256-297.

\vspace{-.1em}\bibitem{roubicek-kruzik-1}
{\rm T. Roub\'\i\v cek, M. Kru\v{z}\'{\i}k}: Microstructure evolution model in micromagnetics. {\it Z. Angew.
Math. Phys.} {\bf 55} (2004), 159--182.

\vspace{-.1em}\bibitem{roubicek-kruzik-2}
{\rm T. Roub\'\i\v cek, M. Kru\v{z}\'{\i}k}: Mesoscopic model for ferromagnets with isotropic hardening.
{\it Z. Angew. Math. Phys.} {\bf 56} (2005), 107--135.

\vspace{-.1em}\bibitem{roubicek-tomassetti}
{\rm T. Roub\'\i\v cek, G. Tomassetti}: Ferromagnets with eddy currents and pinning effects: their thermodynamics and analysis. {\it Math. Models Meth. Appl. Sci. (M3AS)} {\bf 21} (2011), 29-55.

\vspace{-.1em}\bibitem{tartar}
{\rm L. Tartar}: Beyond Young measure. {\it Meccanica} {\bf 30} (1995),
505--526.

\vspace{-.1em}\bibitem{visintin1}
{\rm A. Visintin}: Modified Landau-Lifshitz equation for ferromagnetism.
{\it Physica B} {\bf 233} (1997), 365--369.

\vspace{-.1em}\bibitem{visintin2}
{\rm A. Visintin}: On some models of ferromagnetism.
In: {\it Free Boundary Problems~I} Chiba, 1999. (N.Kenmochi, ed.)
Gakuto Int. Series in Math. Sci. Appl. {\bf 13}, 2000, pp.411--428.
\end{thebibliography}
\end{document}